\documentclass[11pt,reqno,a4paper]{amsart}

\usepackage[foot]{amsaddr} % for author affiliations on the first page
\usepackage{amsmath,amsfonts,amsthm,mathrsfs,amssymb}
\usepackage[usenames]{color}

\usepackage{enumerate}
\usepackage{hyperref}
\usepackage{xcolor}
\hypersetup{
  colorlinks,
  linkcolor={blue!80!black},
  urlcolor={blue!80!black},
  citecolor={blue!80!black}
}
\usepackage{doi}
\usepackage[numbers,square,longnamesfirst]{natbib}

\usepackage{graphicx}
\usepackage{enumerate}
\usepackage{bm}

\usepackage{siunitx} % For units in numbers
\usepackage{listings}
\usepackage[framed,numbered]{matlab-prettifier} % For code formatting
\usepackage{lmodern}

\usepackage[capitalise,noabbrev]{cleveref}
\numberwithin{equation}{section}

\theoremstyle{plain}
\newtheorem{theorem}{Theorem}[section]
\newtheorem*{theorem*}{Theorem}
\newtheorem{lemma}[theorem]{Lemma}

\newtheorem{proposition}[theorem]{Proposition}
\newtheorem*{proposition*}{Proposition}

\newtheorem*{conjecture*}{Conjecture}

\newtheorem*{assumption*}{Assumption}

\theoremstyle{definition}
\newtheorem{definition}[theorem]{Definition}
\newtheorem{example}[theorem]{Example}
\newtheorem{algorithm}{Algorithm}[]
\newtheorem{experiment}{Experiment}[]

\theoremstyle{remark}
\newtheorem{remark}[theorem]{Remark}
\newtheorem*{remark*}{Remark}
\newcommand{\abs}[1]{\left\lvert #1 \right\rvert}

\newcommand{\R}{\mathbb{R}}

\begin{document}
\title[Network area reconstruction]{Blockage detection in networks:
  the area reconstruction method}

\author[E.~Bl{\aa}sten]{Emilia Bl{\aa}sten$^{1,2,3}$}
\email{emilia.blasten@iki.fi \textnormal{(corresponding author)}}
\author[F.~Zouari]{Fedi Zouari$^1$}
\email{zouari.fedi@gmail.com}
\author[M.~Louati]{Moez Louati$^1$}
\email{mzlouati@gmail.com}
\author[M.S.~Ghidaoui]{Mohamed S. Ghidaoui$^1$}
\email{ghidaoui@ust.hk}

\address{$^1$\textit{Department of Civil and Environmental
    Engineering, Hong Kong University of Science and Technology, Hong
    Kong}}

\address{$^2$\textit{Jockey Club Institute for Advanced Study, Hong
    Kong University of Science and Technology, Hong Kong}}

\address{$^3$\textit{Department of Mathematics and Statistics,
    University of Helsinki, Finland}}

\begin{abstract}
  In this note we present a reconstructive algorithm for solving the
  cross-sectional pipe area from boundary measurements in a tree
  network with one inaccessbile end. This is equivalent to
  reconstructing the first order perturbation to a wave equation on a
  quantum graph from boundary measurements at all network ends except
  one. The method presented here is based on a time reversal boundary
  control method originally presented by Sondhi and Gopinath for one
  dimensional problems and later by Oksanen to higher dimensional
  manifolds. The algorithm is local, so is applicable to complicated
  networks if we are interested only in a part isomorphic to a
  tree. Moreover the numerical implementation requires only one matrix
  inversion or least squares minimization per discretization point in
  the physical network. We present a theoretical solution existence
  proof, a step-by-step algorithm, and a numerical implementation
  applied to two numerical experiments.
\end{abstract}


\maketitle

\section{Introduction}

We present a reconstruction algorithm and its numerical implementation
for solving for the pipe cross-sectional area in a water supply
network from boundary measurements. Mathematically this is modelled by
the frictionless waterhammer equations on a quantum graph, with
Kirchhoff's law and the law of continuity on junctions. The
waterhammer equations on a segment $P = {({0,\ell})}$ are given by
\cite{Ghidaoui2005,Wylie1993}
\begin{align}
  &\partial_t H(t,x) = - \frac{a^2(x)}{g A(x)} \partial_x Q(t,x), &&
  t\in\R, x\in P,\label{introWH1}\\ &\partial_t Q(t,x) = - g A(x)
  \partial_x H(t,x), && t\in\R, x\in P,\label{introWH2}
\end{align}
where
\begin{itemize}
\item $H$ is the hydraulic pressure (piezometric head) inside the pipe
  $P$, with dimensions of length,
\item $Q$ is the pipe discharge or flow rate in the direction of
  increasing $x$. Its dimension is length$^3$/time,
\item $a$ is the wave speed in length/time,
\item $g$ is the graviational acceleration in length/time$^2$,
\item $A$ is the pipe's internal cross-sectional area in length$^2$.
\end{itemize}
We model the network by a set of segments whose ends have been joined
together. The segments model pipes. In this model we first fix a
positive direction of flow on each pipe $P_j$, and set the coordinates
${({0,\ell_j})}$ on it. Then we impose \cref{introWH1,introWH2} on the
segments. On the vertices, which model junctions, we require that
$H(t,x)$ has a unique limit no matter which direction the point $x$
tends to the vertex. Moreover on any vertex $V$ we require that
\begin{equation} \label{kirchhoff}
  \sum_{P_j \text{ connected to } V} \nu_j Q_j = 0
\end{equation}
where $\nu_j \in \{+1,-1\}$ gives the \emph{direction of the internal
  normal vector in coordinates $x$} to pipe $P_j$ at $V$, and $Q_j$ is
the boundary flow of pipe $P_j$ at vertex $V$. In other words $\nu_j
Q_j$ is the flow \emph{into} the pipe $P_j$ at the vertex $V$. Hence
the sum of the flows into the pipes at each vertex must be zero.
Therefore there are no sinks or sources at the junctions, and also
none in the network in general, except for the ones at the boundary of
the network that are used to create our input flows for the
measurements.

\medskip
The inverse problem we are set to solve is the following one. We
assume that the wave speed $a(x)=a$ is constant. If not, see
\cref{discussion}. Given a tree network with $N+1$ ends
$x_0,x_1,\ldots, x_N$, we can set the flow and measure the pressure on
these points except for $x=x_0$. Using this information, can we deduce
$A(x)$ for $x$ inside the network?

The problem of finding pipe areas arise in systems such as water
supply networks and pressurized sewers, due to the formation of
blockages during their lifetime \cite{area1,Tolstoy2010}. In water
supply pipes, blockages increase energy consumption and increase the
potential of water contamination. Blockages in sewer pipes increase
the risk of overflows in waste-water collection systems and impose a
risk to public health and environment. Detection of these anomalies
improves the effectiveness of pipe replacement and maintenance, and
hence improves environmental health.

An analogous and important problem is fault detection in electric
cables and feed networks. This problem can be solved in the same way
as blockage detection because of the analogy between waterhammer
equations and Telegrapher’s equations \cite{Jing2018}.

Mathematically the simplest network is a segment joining $x_1$ to
$x_0$. In this setting, the problem formulated as above was solved by
\cite{Sondhi--Gopinath}. Various other algorithms were developed for
solving the one-segment problem both before and after. See
\cite{Bruckstein1985} for a review. After a while some mathematicians
started focusing on the network situation and others into higher
dimensional manifolds. Others continued on tree networks, which are
also the setting of this manuscript. For tree networks Belishev's
boundary control method \cite{Belishev-network-paper1,
  Belishev-network-paper2} showed that a tree network can be
completely reconstructed after having access to all boundary
points. These results were improved more recently by Avdonin and
Kurasov \cite{Kurasov--tree}. A unified approach that uses Carleman
estimates and is applicable to various equations on tree networks was
introduced by Baudouin and Yamamoto \cite{Baudouin--Yamamoto}. Also,
\cite{LeafPeeling} implies that it is possible to solve the problem in
the presence of various different boundary conditions when doing the
measurements. A network having loops presents various challenges to
solving for the cross-sectional area from boundary measurements
\cite{Belishev--loops}. In \cite{Kurasov--tree} the authors
furthermore show that all but one boundary vertex is enough
information, and also that if the network topology is known a-priori,
then the simpler backscattering data is enough to solve the inverse
problem. This latter data is easier to measure: the pressure needs to
be measured only at the same network end from which the flow pulse is
sent.

Our paper has two goals: 1) to solve the inverse problem by a
reconstruction algorithm, and 2) to have a simple and concrete
algorithm that is easy to implement on a computer. In other words we
focus on the reconstruction aspect of the problem, given the full
necessary data, and show that it is possible to build an efficient
numerical algorithm for reconstruction. From the point of view of
implementing the reconstruction, the papers cited above are of various
levels of difficulty, and we admit that this is partly a question of
experience and opinion. We view that apart from
\cite{Sondhi--Gopinath}, all of them focus much on the important
theoretical foundations, but lack in the clarity of algorithmic
presentation for non-experts. Therefore one of the major points of
this text is a clearly defined algorithm whose implementation does not
require understanding its theoretical foundations. See
\cref{algoSect}.

Our method is based on the ideas of \cite{Sondhi--Gopinath,Oksanen}
and is a continuation of our investigation of the single pipe case
\cite{area1}. In both of them the main point is to use \emph{boundary
  control} to produce a wave that would induce a constant pressure in
a \emph{domain of influence} at a certain given time. A simple
integration by parts gives then the total volume covered by that
domain of influence. \emph{Time reversal} is the main tool which
allows us to build that wave without knowing a-priori what is inside
the pipe network.

A numerical implemetation and proof of a working regularization scheme
for \cite{Oksanen} in one space dimension was shown in
\cite{Oksanen--regularization}. Earlier, we studied
\cite{Sondhi--Gopinath} and tested it numerically
in the context of blockage detection in water supply pipes
\cite{area1} and also experimentally. We also showed that the method can be further
extended to detect other types of faults such as leaks
\cite{area3}. The one space-dimensional inversion algorithm can be
implemented more simply than in
\cite{Oksanen,Oksanen--regularization}, even in the case of
networks. This is the main contribution of this paper. Furthermore we
provide a step-by-step numerical implementation and present numerical
experiments. Our reconstruction algorithm is local and time-optimal.
In other words if we are interested in only part of the network, we
can do measurements on only part of the boundary. Furthermore the
algorithm uses time-measurements just long enough to recover the part
of interest. Intuitively, to reconstruct the area $A(x)$ at a location
$x$, the wave must reach the location $x$ and reflects back to the
measurement location. Any measurement done on a shorter time-interval
will not be enough to recover it.

\section{Governing Equations}
\subsection*{Networks}
Denote by $\mathbb G$ a tree network. Let $\mathbb J$ be the set of
internal junction points and $\mathbb P$ the set of pipes, or
segments. The boundary $\partial\mathbb G$ consists of all ends of
pipes that are not junctions of two or more pipes. Each of them
belongs to a single pipe unlike junction points which belong to three
or more. There are no junction points that are connected to exactly
two pipes: these two pipes are considered as one, and the point
between is just an ordinary internal point.

Each pipe $P\in\mathbb P$ is modelled by a segment $({0,\ell})$ where
$\ell$ is the length of the pipe. This defines a direction of positive
flow on the pipe, namely from $x=0$ towards $x=\ell$. The pipes are
connected to each other by junction points. Write $(x,P) \sim v$ if
$v\in\mathbb J$, $x$ is the beginning ($x=0$) or endpoint ($x=\ell$)
of pipe $P\in\mathbb P$, and the latter is connected to $v$ at the
beginning ($x=0$) or end ($x=\ell$), respectively.

\begin{example} \label{ex1}
  \begin{figure}
    \centering
    \includegraphics{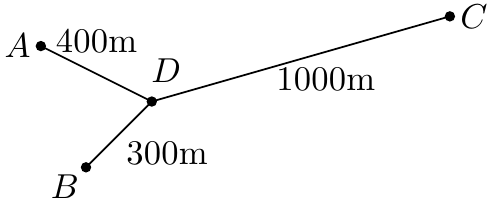}
    \caption{A simple network.}
    \label{fig1}
  \end{figure}
  An example network is depicted in \cref{fig1}. Here $\mathbb J =
  \{D\}$, $\partial\mathbb G = \{A,B,C\}$, $\mathbb P = \{AD, BD,
  DC\}$. Moreover here is a coordinate representation
  \begin{align*}
    AD = ({0,\SI{400}{\m}})&, &BD = ({0,\SI{300}{\m}})&, &DC =
    ({0,\SI{1000}{\m}})&, \\ (0, AD) \sim A&, & (0, BD) \sim B&, & (0,
    DC) \sim D&,\\ (\SI{400}{\m}, AD) \sim D&, & (\SI{300}{\m}, BD)
    \sim D&, & (\SI{1000}{\m}, DC) \sim C&.
  \end{align*}
\end{example}

\begin{example} \label{ex2}
  \cref{ex1} can be implemented numerically as follows. In total there
  are four vertices and three pipes. If vector indexing starts with
  $1$ we number points as $C=1$, $A=2$, $B=3$, $D=4$ and pipes as
  $AD=1$, $BD=2$, $DC=3$. The adjacency matrix \verb#Adj# is defined
  so as pipe number \verb#i# goes from vertex number \verb#Adj(i,1)#
  to vertex number \verb#Adj(i,2)#.
\begin{lstlisting}
  V = 4; % number of vertices
  L   = [400; 300; 1000]; % length of pipes
  Adj = [2 4; 3 4; 4 1]; % adjacency matrix
\end{lstlisting}
\end{example}

\subsection*{Equations}
Inside pipe $P$ the perturbed pressure head $H$ and cross-sectional
discharge $Q$ satisfy
\begin{align}
  &\partial_t H(t,x) = - \frac{a^2}{g A(x)} \partial_x Q(t,x), &&
  t\in\R, \quad x\in P,\label{WH1}\\ &\partial_t Q(t,x) = - g A(x)
  \partial_x H(t,x), && t\in\R, \quad x\in P,\label{WH2}
\end{align}
where $a,A,g$ denote the wave speed, cross-sectional area and
gravitational acceleration. The sign of $Q$ is chosen so that $Q>0$
means the flow goes from $0$ to $\ell$. In \cref{ex1} the positive
flow goes from $A$ to $D$, from $B$ to $D$ and from $D$ to $C$. The
wave speed is assumed constant.

The pressure is a scalar: if $v\in\mathbb J$ connects two or more
pipes $(x_j,P_j) \sim v$, $j=1,2,\ldots$, then $H$ has a unique value
at $v$
\begin{equation} \label{scalarH}
  \lim_{\substack{x\to x_j\\x\in P_j}} H(t,x) = \lim_{\substack{x\to
      x_k\\x\in P_k}} H(t,x), \qquad t\in\R, \quad (x_j,P_j) \sim v
  \sim (x_k,P_k).
\end{equation}
The flow satisfies mass conservation, i.e. a condition analogous to
Kirchhoff's law. The total flow into a junction must be equal to the
total flow out of the junction at any time. To state this as an
equation we define the \emph{internal normal vector} $\nu$ for any
pipe end. If the pipe $P$ has coordinate representation ${(0,\ell)}$
then $\nu(0) = +1$ and $\nu(\ell) = -1$. Recall that $Q>0$ mean a
positive flow from the direction of $0$ to $\ell$. This means that
$\nu(x)Q(t,x)$ is the flow \emph{into} the pipe at point
$x\in\{0,\ell\}$. If it is positive then there is a net flow of water
into $P$ through $x$. If it is negative then there is a net flow out
of $P$ through $x$. Mass conservation is then written as
\begin{equation} \label{Kirchhoff}
  \sum_{(x,P)\sim v} \nu(x) Q(t,x) = 0, \qquad t\in\R, \quad
  v\in\mathbb J.
\end{equation}

\subsection*{Initial conditions}
The model assumes unperturbed initial conditions
\begin{equation}\label{zeroInitial}
  H(t,x)=Q(t,x)=0, \qquad t<0, \quad x\in P
\end{equation}
for any pipe $P\in\mathbb P$.

\subsection*{Direct problem}
In this section we define the behaviour of the pressure $H$ and pipe
cross-sectional discharge $Q$ in the network given a boundary flow and
network structure. This is the direct problem. Recall that we have one
inaccessible end of the network, $x_0$, whose boundary condition must
be an inactive one, i.e. must not create waves when there are no
incident waves. To make the problem mathematically well defined we
choose for example \cref{x0BC}. Other options would work too, for
example involving derivatives of $H$ or $Q$, but it is questionable
how physically realistic such an arbitrary boundary condition is. The
main point is that any inactive condition at the inaccessible end is
allowed without risking the reconstruction. This is because the
theoretical waves used in the calculations of our reconstruction
algorithm never have to reach this final vertex, and so the boundary
condition there never has a chance to modify reflected waves.
\begin{definition} \label{model}
  We say that $H$ and $Q$ \emph{satisfy the network wave model with
    boundary flow $F : \R\times\partial\mathbb G\setminus\{x_0\} \to
    \R$} if $F(t,x)=0$ for $t<0$ and $H,Q$ satisfy \cref{WH1,WH2}, the
  junction conditions of \cref{scalarH,Kirchhoff} and the initial
  conditions of \cref{zeroInitial}. Furthermore $Q$ must satisfy the
  boundary conditions
  \begin{align}
    &\nu(x)Q(t,x)=F(t,x), &&\qquad x\in\partial\mathbb G, x\neq x_0,
    t\in\R \\ &A(t) Q(t,x) + B(t) H(t,x) = 0, &&\qquad x=x_0,
    t\in\R \label{x0BC}
  \end{align}  
  for some given functions or constants $A(t),B(t)$ that we do not
  need to know.
\end{definition}
\begin{remark*}
  Note that $\nu Q = F$ implies that $F$ is the flow \emph{into} the
  network. If $F>0$ fluid enters, and if $F<0$ fluid is coming out.
\end{remark*}

Let us mention a few words about the unique solvability of the network
wave model with a given boundary flow $F$. First of all we have not
given any precise function spaces where the coefficients of the
equation or the boundary flow would belong to. This means that it is
not possible to find an exact reference for the solvability. On the
other hand this is not a problem for linear hyperbolic problems in
general.

The problem is a one-dimensional linear hyperbolic problem on various
segments with co-joined boundary conditions. As the waves propagate
locally in time and space, one can start with the solution to a wave
equation on a single segment, as in e.g. Appendix 2 to Chapter V in
\cite{Courant--Hilbert2}. Then when the wavefront approaches a
junction, \cref{scalarH,Kirchhoff} determine the transmitted and
reflected waves to each segment joined there. Then the wave propagates
again according to \cite{Courant--Hilbert2}, and the boundary
conditions are dealt with as in a one segment case. At no point is
there any space to make any ``choices'', and thus there is unique
solvability. We will not comment on this further, but for more
technical details we refer to \cite{Courant--Hilbert2} for the wave
propagation in a segment and the boundary conditions, and to Section 3
of \cite{Belishev-network-paper1} for the wave propagation through
junctions. For an efficient numerical algorithm to the direct problem,
see \cite{Karney--McInnis}.

\subsection*{Boundary measurements}
The area reconstruction method presented in this paper requires the
knowledge of the \emph{impulse-response matrix} (IRM) for all boundary
points except one, which we denote by $x_0$.
\begin{definition}
  We define the \emph{impulse-response matrix}, or IRM, by $K =
  (K_{ij})_{i,j=1}^N$. For a given $i$ and $j$ we assume that $H$ and
  $Q$ satisfy the network wave model with boundary
  flow\footnote{$\delta_0(t)$ has dimensions of time$^{-1}$ because
    $\int \delta_0(t) dt = 1$ and $dt$ has dimensions of time.}
  \begin{equation}\label{boundaryFlow}
    \nu(x) Q(t,x) = \begin{cases}
      V_0 \, \delta_0(t), & x=x_i\\ 0, &
      x\neq x_i
    \end{cases}
  \end{equation}
  for $t\in\R$ and $x\in\partial\mathbb G, x\neq x_0$. Here $V_0$ is
  the volume of fluid injected at $t=0$. Then we set
  \begin{equation}\label{impulseResponseMatrix}
    K_{ij}(t) = H(t,x_j) / V_0
  \end{equation}
  for any $t\in\R$.
\end{definition}

The index $i$ represents the source and $j$ the receiver. Note also
that the IRM gives complete boundary measurement information: if
$\nu(x) Q(t,x_i) = F(t,x_i)$ were another set of injected flows at the
boundary, then the corresponding boundary pressure would be given by
\begin{equation} \label{HfromIRM}
  H(t,x_j) = \sum_{\substack{x_i\in\partial\mathbb G\\x_i\neq x_0}}
  \int K_{ij}(t-s,x_i) F(s,x_i) ds
\end{equation}
by the principle of superposition.

\section{Area reconstruction algorithm}
\subsection*{Strategy}
Once the impulse-response matrix from \cref{impulseResponseMatrix} has
been measured, we have everything needed to determine the
cross-sectional area in a tree network. As in the one pipe case
\cite{area1}, we will calculate special ``virtual'' boundary
conditions, which if applied to the pipe system, would make the
pressure constant in a given region at a given time. Exploiting this
and the knowledge of the total volume of water added to the network by
these virtual boundary conditions gives the total volume of the
region. Slightly perturbing the given region reveals the
cross-sectional area.

Multiply \cref{WH1} by $gA/a^2$ and integrate over a time-interval
${({0,\tau})}$, for a fixed $\tau>0$, and the whole network $\mathbb
G$. This gives
\begin{equation} \label{impedanceByBoundaryData}
  \sum_{x_j\in\partial\mathbb G} \int_0^\tau \nu(x_j) Q(t,x_j) dt =
  \int_{\mathbb G} H(\tau,x) \frac{g A(x)}{a^2(x)} dx
\end{equation}
if $H(0,x)=0$ and mass conservation from \cref{Kirchhoff} applies. Let
$p\in\mathbb G$ be a point at which we would like to recover the
cross-sectional area. To it we associate a set $D_p\subset\mathbb G$
which we shall define precisely later. Let us assume that there are
boundary flows $Q_p(t,x_j)$ so that at time $t=\tau$ we have
\begin{equation}\label{Hcharacteristic}
  H(\tau,x) = \begin{cases}
    h_0, &x\in D_p,\\
    0, &x\notin D_p,
  \end{cases}
\end{equation}
for some fixed pressure $h_0$. Then from
\cref{impedanceByBoundaryData} we have
\[
\sum_{x_j\in\partial\mathbb G} \int_0^\tau \nu(x_j) Q_p(t,x_j) dt =
\frac{h_0g}{a^2} \int_{D_p} A(x) dx.
\]
Denote the integral on the left by $V(D_p)$. We can calculate its
values once we know $Q_p$, hence we know the value of the integral on
the right. By varying the shape of $D_p$ we can then find the area
$A(p)$.

\subsection*{Admissible sets} \label{Dsection}
The only requirement for $D_p$ in the previous section was that
\cref{Hcharacteristic} holds. Boundary control, e.g. as in
\cite{Belishev-network-paper2}, implies that there are suitable
boundary flows $Q_p$ such that the equation holds for any reasonable
set $D_p$. However it is not easy to calculate the flows given an
arbitrary $D_p$. In this section we define a class of such sets for
which it is very simple to calculate the flows.

Let $p\in\mathbb G \setminus \mathbb J$ be a non-junction point in the
network at which we wish to solve for the cross-section area $A(p)$.
Since we have a tree network the point $p$ splits $\mathbb G$ into two
networks. Let $D_p$ be the part that is not connected to the
inaccessible boundary point $x_0$. The boundary of $D_p$ consists of
let us say $y_1,y_2,\ldots,y_k$ and $p$, where the points $y_j$ are
also boundary points of the original network.

For each boundary point $y_j\in\partial D_p$, $y_j\neq p$ define the
action time
\begin{equation}\label{actionTime}
  f(y_j) = \operatorname{TT}(y_j,p)
\end{equation}
where $\operatorname{TT}$ gives the travel-time of waves from $y_j$ to
$p$ calculated along the shortest path in $D_p$. Then set $f(x_j)=0$
for $x_j\in\partial\mathbb G \setminus\partial D_p$.

\begin{definition} \label{admissibleSet}
  We say that $D_p$ is an \emph{admissible set} associated with
  $p\in\mathbb G \setminus \mathbb J$ and with \emph{action time} $f$,
  if $D_p$ and $f$ are defined as above given $p$.
\end{definition}

It turns out that with the choice of admissible sets made above we
have
\begin{equation}\label{DpFromActionTimes}
  D_p = \{ x \in \mathbb G \mid \operatorname{TT}(x_j,x) < f(x_j)
  \text{ for some } x_j\in\partial\mathbb G \}.
\end{equation}
This is because $D_p$ lies between $p$ and the boundary points $x_j$
at which $f(x_j)\neq0$. This gives a geometric interpretation to the
set $D_p$, i.e. that it is the \emph{domain of influence} of the
\emph{action times $f(x_j)$, $x_j\in\partial\mathbb G$}. If we would
have zero boundary flows at first, and then active boundary flows when
$x_j\in\partial\mathbb G$, $\tau-f(x_j)< t\leq\tau$, then the
transient wave produced would have propagated throught the whole set
$D_p$ at time $t=\tau$ but not at all into $\mathbb G\setminus D_p$.

\subsection*{Reconstruction formula for the area} \label{areaSection}
Now that the form of the admissible sets have been fixed, we are ready
to prove in detail what was introduced at the beginning of this
section, namely a formula for solving the unknown pipe cross-sectional
area. For simplicity we assume that the wave speed is constant.

\begin{theorem} \label{areaAtP}
  Let $p\in \mathbb G \setminus \mathbb J$ be a non-junction point and
  $D_p$, $f$ the associated admissible set and action time. Let
  $\tau>\max f$. For a small time-interval $\Delta_t>0$ set
  \[
  (f+\Delta_t)(x_j) = \begin{cases} f(x_j)+\Delta_t, &x_j \in \partial
    D_p \setminus \{p\},\\ 0, &x_j \in \partial\mathbb G \setminus
    \partial D_p. \end{cases}
  \]
  For $\phi=f$ or $\phi=f+\Delta_t$ denote
  \begin{equation} \label{DpdDef}
    D^\phi = \{ x\in\mathbb G \mid \operatorname{TT}(x,x_j) <
    \phi(x_j) \text{ for some } x_j\in \partial D_p \setminus \{p\} \}
  \end{equation}
  so $D^f = D_p$ and $D^{f+\Delta_t}$ is a slight expansion of the
  former.

  Assume that $H_\phi,Q_\phi$ satisfy the network wave model with a
  boundary flow $F$ which is nonzero only during the action time
  $\phi$, namely $F(t,x_j)=0$ when $x_j\in\partial\mathbb G$, $x_j\neq
  x_0$ and $0 \leq t \leq \tau - \phi(x_j)$. Finally, assume that
  \begin{equation}
    H_\phi(\tau,x) = \begin{cases} h_0, &x\in D^\phi,\\ 0, &x\notin
      D^\phi,\end{cases}
  \end{equation}
  for some given pressure head $h_0>0$ at time $t=\tau$.

  Denote by
  \[
  V(\phi,\tau) = \frac{a^2}{h_0g} \sum_{x_j\in \partial D_p
    \setminus\{p\}} \nu(x_j) \int_0^\tau Q_\phi(t,x_j) dt
  \]
  the total volume of fluid injected into the network from the
  boundary in the time-interval ${({0,\tau})}$ to create the waves
  $H_\phi,Q_\phi$. Then
  \begin{equation}
    A(p) = \lim_{\Delta_t\to0} \frac{V(f+\Delta_t,\tau) -
      V(f,\tau)}{a\Delta_t}
  \end{equation}
  gives the cross-sectional area of the pipe at $p$.
\end{theorem}
\begin{proof}
  The assumption about the boundary flow implies that
  \begin{equation}
    H_\phi(t,x_j) = Q_\phi(t,x_j) = 0
  \end{equation}
  in the same space-time set, i.e. when $x_j\in\partial\mathbb G$,
  $x_j\neq x_0$ and $0 \leq t \leq \tau - \phi(x_j)$. Consider
  \cref{WH1} for $H_\phi,Q_\phi$ where $\phi=f$ or $\phi=f+\Delta_t$
  is fixed. Multiply the equation by $gA/a^2$ and integrate
  $\int_0^\tau \int_{\mathbb G} \ldots dx dt$. This gives
  \[
  - \int_0^\tau \int_{\mathbb G} \partial_x Q_\phi(t,x) dx dt =
  \int_0^\tau \int_{\mathbb G} \frac{g A(x)}{a^2} \partial_t
  H_\phi(t,x) dx dt.
  \]
  The right-hand side is equal to
  \[
  \frac{g}{a^2} \int_{\mathbb G} A(x) \big(H_\phi(\tau,x) -
  H_\phi(0,x) \big) dx = \frac{g h_0}{a^2} \int_{D^\phi} A(x) dx
  \]
  because $H_\phi(0,x)=0$ by \cref{zeroInitial}. We will use the
  junction conditions of \cref{Kirchhoff} to deal with the left-hand
  side. But before that let us use a fixed coordinate system of the
  network $\mathbb G$.

  Let the pipes of the network be $P_1,\ldots, P_n$ and model them in
  coordinates by the segments ${({0,\ell_1})}, \ldots,
  {({0,\ell_n})}$, respectively. On pipe $P_k$, denote by $H_{\phi,k}$
  the scalar pressure head, and by $Q_{\phi,k}$ the pipe discharge
  into the positive direction. Then
  \[
  \int_{\mathbb G} \partial_x Q_\phi(t,x) dx = \sum_{k=1}^n
  \int_0^{\ell_k} \partial_x Q_{\phi,k}(t,x) dx = \sum_{k=1}^n \big(
  Q_{\phi,k}(t,\ell_k) - Q_{\phi,k}(t,0) \big).
  \]
  Note that $Q_{\phi,k}(t,\ell_k)$ is simply the discharge \emph{out}
  of the pipe $P_k$ at the latter's endpoint represented by $x=\ell_k$
  at time $t$. Similarly $-Q_{\phi,k}(t,0)$ is the discharge out from
  the other endpoint, the one represented by $x=0$. Now
  \cref{Kirchhoff} implies after a few considerations that
  \[
  -\int_{\mathbb G} \partial_x Q_\phi(t,x) dx = \sum_{x_j\in \partial
    D_p \setminus \{p\}} \nu(x_j) Q_\phi(t,x_j).
  \]
  Namely, previously we saw that the integral is equal to the sum of
  the total discharge out of every single pipe. But the discharge out
  of one pipe must go \emph{into} another pipe at junctions (there are
  no internal sinks or sources). Hence the discharges at the junctions
  cancel out, and only the ones at the boundary $\partial\mathbb G$
  remain. The boundary values are zero on $\partial \mathbb G
  \setminus (\partial D_p \cup \{x_0\})$ by the definition of $f$ and
  $f+\Delta_t$. We also have zero initial values and a non-active
  boundary condition at $x_0$, hence $Q_\phi(t,x_0)=0$ too. Thus we
  have shown that
  \begin{equation}
  \int_{D^\phi} A(x) dx = \frac{a^2}{g h_0} \sum_{x_j\in \partial D_p
    \setminus \{p\}} \nu(x_j) Q_\phi(t,x_j) = V(\phi,\tau).
  \end{equation}
  
  Next, we will show that
  \begin{equation} \label{dV}
    V(f+\Delta_t,\tau) - V(f,\tau) = \int_{D^{f+\Delta_t}\setminus
      D^f} A(x) dx.
  \end{equation}
  Recall that $p$ is not a vertex. Hence it is on a unique pipe, let's
  say ${({0,\ell})}$ and $p$ is represented by $x_p$ on this
  segment. Furthermore assume that (or change coordinates so that) the
  point represented by $0$ is in $D_p$, and the one by $\ell$ is
  not. This implies that the difference of sets
  $D^{f+\Delta_t}\setminus D^f$ is just a small segment on
  ${({0,\ell})}$.

  Let us look at the effect of $\Delta_t$ on $D^{f+\Delta_t}$ which is
  defined in \cref{DpdDef}. We have $x\in D^{f+\Delta_t} \setminus
  D^f$ if and only if $\operatorname{TT}(x,x_j) < f(x_j) + \Delta_t$
  for some boundary point $x_j\in\partial D_p$, $x_j\neq p$, but also
  $\operatorname{TT}(x,x_k) \geq f(x_k)$ for all boundary points
  $x_k\in\partial D_p$, $x_k\neq p$. The former implies that $x$ is at
  most travel-time $\Delta_t$ from $D_p$, and the latter says that it
  should not be in $D_p$. In other words $D^{f+\Delta_t}\setminus D^f
  = \{x\in {({0,\ell})} \mid x_p \leq x < x_p+a\Delta_t \}$ and so
  \begin{equation} \label{dVformula}
    V(f+\Delta_t,\tau)-V(f,\tau) = \int_{x_p}^{x_p+a\Delta_t} A(x) dx,
  \end{equation}
  where we abuse notation and denote the cross-sectional area of the
  pipe modelled by ${({0,\ell})}$ at the location $x$ also by $A(x)$
  without emphasizing that it is the area on this particular model of
  this particula pipe.

  \Cref{dVformula} gives the area at $p$, $A(x_p)$, by
  differentiating. Let $B(s) = \int_{x_p}^{x_p+s} A(x) dx$. Then
  $A(x_p) = \partial_s B(0)$ and the right-hand side of
  \cref{dVformula} is equal to $B(a\Delta_t)$. The chain rule for
  differentiation gives
  \[
  A(x_p) = \partial_s B(s)_{|s=0} = \frac{1}{a} \partial_{\Delta_t}
  (B(a\Delta_t))_{|\Delta_t=0} = \frac{1}{a} \lim_{\Delta_t\to0}
  \frac{V(f+\Delta_t,\tau)-V(f,\tau)}{\Delta_t}
  \]
  from which the claim follows.
\end{proof}

\subsection*{Solving for the area} \label{1pipeIdea}
In this section we will show one way in which boundary values of $Q$
can be determined so that the assumptions of \cref{areaAtP} are
satisfied. It is previously known that there are boundary flows giving
\cref{Hcharacteristic}, for example by \cite{Belishev-network-paper2}
where the authors show \emph{the exact $L^2$-controllability} of the
network both locally and in a time-optimal way. However there was no
simple way of calculating such boundary flows and the numerical
reconstruction algorithm of that paper does not seem computationally
efficient as it uses a Gram--Schmidt orthogonalization process on a
number of vectors inversely proportional to the network's
discretization size. We show that if the flow satisfies a certain
boundary integral equation then this is the right kind of
flow. Moreover, in the appendix we give a proof scheme for showing
that the equation has a solution.

More recently various layer-peeling type of methods have appeard
\cite{Kurasov--tree, LeafPeeling}. The method we present here is based
on another idea, one whose roots are in the physics of waves, namely
time reversibility. This is in essence a combination of the one pipe
case originally considered in \cite{Sondhi--Gopinath}, and of a
fundamentally similar idea for higher dimensional manifolds introduced
in \cite{Oksanen}. In the latter the author considers domains of
influence and action times on the boundary, and builds a boundary
integral equation whose solution then reveals the unknown inside the
manifold. This will be our guide.

\bigskip
Let us recall the unique continuation principle used for the area
reconstruction method for one pipe in
\cite{area1,Sondhi--Gopinath}. Consider a pipe of length $\ell>0$,
modelled by the interval ${({0,\ell})}$. Let $\mathcal H$ and
$\mathcal Q$ satisfy the Waterhammer \cref{WH1,WH2} without requiring
any initial conditions. Then
\begin{lemma} \label{uniqCont1pipe}
  If $\mathcal H(t,0)=2h_0$ and $\mathcal Q(t,0)=0$ for $t_m<t<t_M$,
  then inside the pipe we would have $\mathcal H(t,x)=2h_0$ and
  $\mathcal Q(t,x)=0$ in the space-time triangle $x/a +
  \abs{t-(t_M+t_m)/2} < (t_M-t_m)/2$, $0<x<\ell$. See \cref{fig2}.
\end{lemma}

\begin{figure}
  \centering
  \includegraphics{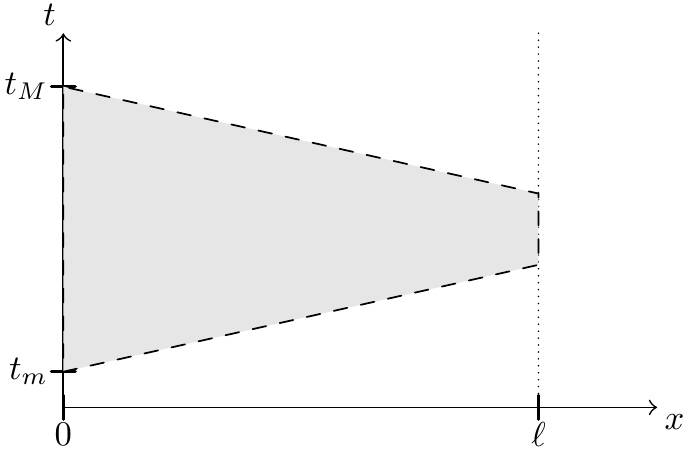}
  \caption{The region $\frac{x}{a} + \abs{t-\frac{t_M+t_m}{2}} <
    \frac{t_M-t_m}{2}$, $0<x<\ell$.}
  \label{fig2}
\end{figure}
The lemma was then used to build a virtual solution $H,Q$ satisfying
also the initial conditions, and which would have $H(\tau,x)=h_0$ for
$x < a \tau$ and $H(\tau,x)=0$ for $x>a\tau$ for a given
$\tau>0$. Without going into detailed proofs, the same unique
continuation idea works for a tree network. The reason is that one can
propagate $\mathcal H=2h_0$, $\mathcal Q=0$ from one end of a pipe to
the other, but by keeping in mind that the time-interval where these
hold shrinks as one goes further into the pipe. Do this first on all
the pipes that touch the boundary. Then use the junction conditions
from \cref{scalarH,Kirchhoff} to see that $\mathcal H=2h_0$, $\mathcal
Q=0$ on the next junctions. Then repeat inductively. We have shown
\begin{proposition}\label{treeUniqCont}
  Let $p\in\mathbb G \setminus \mathbb J$ and let $D_p \subset \mathbb
  G$ be admissible associated with $p$ and with action time $f$. Let
  $\mathcal H, \mathcal Q$ satisfy \cref{WH1,WH2} and the junction
  conditions of \cref{scalarH,Kirchhoff}. Let $\tau \geq \max f$ and
  assume that
  \begin{equation} \label{inductionReq}
    \mathcal H(t,x_j) = 2h_0, \qquad \mathcal Q(t,x_j) = 0, \qquad
    x_j\in\partial\mathbb G, \quad \abs{t-\tau} < f(x_j).
  \end{equation}
  Then $\mathcal H(t,x) = 2h_0$ and $\mathcal Q(t,x) = 0$ whenever
  $x\in\mathbb G$, $0<t<2\tau$ and
  \begin{equation} \label{inductionConc}
    \operatorname{TT}(x,x_j) + \abs{\tau-t} < f(x_j)
  \end{equation}
  for some $x_j\in\partial\mathbb G$.
\end{proposition}

We can now write an integral equation whose solution gives waves with
$H(\tau,x)=h_0$ for $x\in D_p$ and $H(\tau,x)=0$ for $x\notin D_p$ at
time $t=\tau$.
\begin{theorem} \label{treeEquationForH1}
  Let $K_{ij}$ be the impulse-response matrix from
  \cref{impulseResponseMatrix}. If $A$ is constant near each network
  boundary point we have
  \[
  K_{ij}(t) = \frac{a}{A(x_i) g} \delta_0(t) \delta_{ij} + k_{ij}(t)
  \]
  for some function\footnote{$k_{ij}$ might be a distribution if $A$
    is not smooth enough.} $k_{ij}=k_{ji}$ that vanishes near $t=0$.

  Let $p\in\mathbb G \setminus \mathbb J$ be and let $D_p \subset
  \mathbb G$ be the admissible set associated with $p$, and with
  action time $f$. Take $\tau\geq\max f$ and let $Q_p(t,x_j)$ satisfy
  \begin{align}
    h_0 &= \frac{a}{A(x_j) g} \nu(x_j) Q_p(t,x_j) \notag\\ &\quad +
    \sum_{\substack{x_i\in\partial\mathbb G\\x_i\neq x_0}}
    \frac{\nu(x_i)}{2} \int_0^\tau Q_p(s,x_i) \big( k_{ij}(\abs{t-s})
    + k_{ij}(2\tau-t-s) \big) ds \label{Qequation}
  \end{align}
  when $x_j\in\partial\mathbb G$, $x_j\neq x_0$,
  $\tau-f(x_j)<t\leq\tau$ and
  \begin{equation} \label{Qsupport}
    Q_p(t,x_j) = 0
  \end{equation}
  when $x_j\in\partial\mathbb G$, $x_j\neq x_0$ and $0\leq t \leq
  \tau-f(x_j)$.
  
  Then if $H,Q$ satisfy the network wave model with boundary flow $\nu
  Q_p$ we have
  \begin{equation} \label{H1}
    H(\tau,x) = \begin{cases} h_0, &x\in D_p,\\ 0, &x\in \mathbb G
      \setminus D_p. \end{cases}
  \end{equation}
\end{theorem}
\begin{remark}
  If one sets $Q_p(2\tau-t,x_j)=Q_p(t,x_j)$, one could instead have
  \[
  h_0 = \frac{a}{A(x_j) g} \nu(x_j) Q_p(t,x_j) +
  \sum_{\substack{x_i\in\partial\mathbb G\\x_i\neq x_0}}
  \frac{\nu(x_i)}{2} \int_0^{2\tau} Q_p(s,x_i) k_{ij}(\abs{t-s}) ds
  \]
  with $x_j\in\partial\mathbb G$, $\abs{t-\tau}<f(x_j)$ and
  $Q_p(t,x_j)=0$ for $\abs{t-\tau}\geq f(x_j)$.
\end{remark}
\begin{proof}
  The claim for the impulse response matrix is standard. See for
  example Appendix 2 to Chapter V in
  \cite{Courant--Hilbert2} for the solution to the one
  segment setting, and then use mathematical induction and the
  junction conditions of \cref{scalarH,Kirchhoff}.

  Extend $Q_p$ symmetrically past $\tau$,
  i.e. $Q_p(2\tau-t,x_j)=Q_p(t,x_j)$ for $0\leq t\leq\tau$ and
  $x_j\in\partial\mathbb G$, $x_j\neq x_0$. Continue $H$ and $Q$ to
  $0<t<2\tau$ while still having $Q=Q_p$ as the boundary condition at
  $x\neq x_0$, and \cref{x0BC} when $x=x_0$. Define
  \[
  \mathcal H(t,x) = H(t,x) + H(2\tau-t,x), \qquad \mathcal Q(t,x) =
  Q(t,x) - Q(2\tau-t,x)
  \]
  for $x\in\mathbb G$ and $0<t<2\tau$. The symmetry of $Q_p$ implies
  that $\mathcal Q(t,x_j) = 0$ for $x_j\in\partial\mathbb G$, $x_j\neq
  x_0$ and $0<t<2\tau$.

  For the pressure, recall that the properties of the impulse response
  matrix from \cref{HfromIRM} imply that
  \[
  H(t,x_j) = \sum_{\substack{x_i\in\partial\mathbb G\\x_i\neq x_0}}
  \int_{-\infty}^\infty K_{ij}(t-s) \nu(x_i) Q_p(s,x_i) ds
  \]
  for $x_j\in\partial\mathbb G$, $x_j\neq x_0$ and $0<t<2\tau$. It is
  then easy to calculate that
  \[
  H(t,x_j) = \frac{a}{A(x_j) g} \nu(x_j) Q_p(t,x_j) +
  \sum_{\substack{x_i\in\partial\mathbb G\\x_i\neq x_0}} \int_0^t
  k_{ij}(t-s) \nu(x_i) Q_p(s,x_i) ds
  \]
  because $K_{ij}(t-s) = 0$ when $s>t$ and $Q_p(s,x_i)=0$ when
  $s<0$. For $H(2\tau-t,x_j)$ split the integral to get
  \begin{align*}
    &\int_0^{2\tau-t} k_{ij}(2\tau-t-s) Q_p(s,x_i) ds \\&\qquad =
    \int_0^\tau k_{ij}(2\tau-t-s) Q_p(s,x_i) ds + \int_\tau^{2\tau-t}
    k_{ij}(2\tau-t-s) Q_p(s,x_i) ds \\&\qquad = \int_0^\tau
    k_{ij}(2\tau-t-s) Q_p(s,x_i) ds + \int_t^\tau k_{ij}(s-t)
    Q_p(s,x_i) ds
  \end{align*}
  where we again used the time-symmetry of $Q_p$. Summing all terms
  and using $Q_p(2\tau-t,x_j)=Q_p(t,x_j)$ we see that
  \begin{align*}
    &\mathcal H(t,x_j) = H(t,x_j) + H(2\tau-t,x_j) = 2
    \frac{a}{A(x_j) g} \nu(x_j) Q_p(t,x_j) \\&\quad +
    \sum_{\substack{x_i\in\partial\mathbb G\\x_i\neq x_0}} \nu(x_i)
    \int_0^\tau Q_p(s,x_i) \big( k_{ij}(\abs{t-s}) + k_{ij}(2\tau-t-s)
    \big) ds = 2h_0
  \end{align*}
  when $x_j\in\partial\mathbb G$ and $\abs{t-\tau}<f(x_j)$. The
  assumptions of \cref{treeUniqCont} are now satisfied and so
  $\mathcal H(t,x) = 2h_0$ and $\mathcal Q(t,x) = 0$ when $0<t<2\tau$
  and $\operatorname{TT}(x,x_j) + \abs{\tau-t} < f(x_j)$ for some
  $x_j\in\partial\mathbb G$. \cref{DpFromActionTimes} and the finite
  speed of wave propagation imply \cref{H1}.
\end{proof}

\begin{remark}
  Assuming an impulse-response matrix that is measured to infinite
  precision and without modelling errors, one can show that
  \cref{Qequation,Qsupport} have a solution. The idea of the proof is
  shown in the appendix. The solution is not necessarily unique, but
  it gives a unique reconstructed cross-sectional area.
\end{remark}

\section{Step-by-step algorithm} \label{algoSect}
\subsection*{Measuring the impulse-response matrix}
Recall \cref{impulseResponseMatrix}: the impulse-response matrix
$K=(K_{ij})_{i,j=1}^N$ is defined by $K_{ij}(t) = H(t,x_j)/V_0$ where
$H,Q$ solve the Waterhammer \cref{WH1,WH2}, junction
\cref{scalarH,Kirchhoff}, zero initial conditions of
\cref{zeroInitial} and a flow impulse of volume $V_0$ at the boundary
vertex $x_i$ and zero flow at other accessible vertices, as in
\cref{boundaryFlow}. Here $x_i, x_j$ with $i,j\neq0$ are the
accessible boundary points, and $x_0$ is the inaccessible one that
doesn't produce surges.

\subsection*{Solving for the cross-sectional area}
In this second part of the step-by-step reconstruction algorithm we
assume that the impulse-response matrix $K$ has been calculated. It
can be either measured by closing all accessible ends, or by measuring
the system response in a different setting (i.e. different boundary
conditions), then pre-process the measured signal to obtain the
desired matrix $K$.

Once the impulse-response matrix has been measured as discussed in the
previous paragraph, the following mathematical algorithm can be
applied to recover the cross-sectional area inside a chosen pipe or
pipe-segment in the network.

\begin{algorithm} \label{alg1}
  This algorithm calculates the cross-sectional area using a
  discretization and \cref{alg2}.
  \begin{enumerate}
  \item \label{step:responseMat} Define $k_{ij}(t)$ for $i,j\neq0$ by
    \[
    k_{ij}(t) = K_{ij}(t) - \frac{a}{A(x_i) g} \delta_0(t) \delta_{ij}
    \]
    where $\delta_{ij}$ is the Kronecker delta.
  \item Choose a point $p_1$ in the network that is not a
    junction. The algorithm will reconstruct the cross-sectional area
    starting from $p_1$ and going towards $x_0$ until it hits the
    endpoint $p_2$ of the current pipe.
  \item Split the interval between $p_1$ and $p_2$ into pieces of
    length $\Delta_x$.
  \item \label{step:chooseP} Let $p$ be any point between two pieces
    above that has not been chosen yet. Calculate the internal volume
    $V(p)$ using \cref{alg2}.
  \item Redo \cref{step:chooseP} for all the points in the
    discretization of the interval between $p_1$ and $p_2$, and save
    the values of $V(p)$ associated with the point $p$.
  \item Denote the discretization by $(p(0)=p_1, p(1), \ldots,
    p(M)\approx p_2)$. Then the area at $p(k)$ is approximately the
    volume between the points $p(k)$ and $p(k+1)$ divided by
    $\Delta_x$. In other words
    \[
    A\big(p(k)\big) \approx \frac{V\big(p(k+1)\big) -
      V\big(p(k)\big)}{\Delta_x}.
    \]
    This would be an equality if $\Delta_x$ were infinitesimal, or if
    $A(p(k))$ would signify the average area over the interval $(p(k),
    p(k+1))$.
  \end{enumerate}
\end{algorithm}

\begin{algorithm} \label{alg2}
  This algorithm calculates the internal volume of the piece of
  network cut off by $p$: namely all points from which you have to
  pass through $p$ to get to $x_0$.
  \begin{enumerate}
  \item For any boundary point $x_j\neq x_0$ set
    \[
    f(x_j) = \begin{cases} \operatorname{TT}(x_j,p), & \text{if $p$ is
        between $x_j$ and $x_0$}, \\ 0, & \text{if not}, \end{cases}
    \]
    where $\operatorname{TT}$ gives the travel-time between points. It
    is just the distance in the network divided by the wave speed.
  \item \label{step:tau} Take $\tau \geq \max f$ and fix a pressure
    head $h_0>0$.
  \item \label{step:solveQ} For any boundary point $x_j\neq x_0$ and
    time $\tau-f(x_j) < t \leq \tau$, using regularization if
    necessary, let $Q_p$ solve
    \begin{align}
      &h_0 = \frac{a}{A(x_j) g} \nu(x_j) Q_p(t,x_j) \label{eq:solveQ}
      \\ & + \sum_{\substack{x_i\in\partial\mathbb G\\x_i\neq x_0}}
      \frac{\nu(x_i)}{2} \int_0^\tau Q_p(s,x_i) \big(
      k_{ij}(\abs{t-s}) + k_{ij}(2\tau-t-s) \big) ds \notag
    \end{align}
    and also simultanously set
    \[
    Q_p(t,x_i) = 0
    \]
    for boundary points $x_i\neq x_0$ and time $0\leq t \leq \tau -
    f(x_j)$. Thus the integral above can be calculated on $\tau-f(x_j)
    < s \leq \tau$.
  \item \label{step:volume} Set
    \begin{equation} \label{eq:volume}
      V(p) = \frac{a^2}{h_0g} \sum_{x_j\in\partial\mathbb G}
      \int_0^\tau \nu(x_j) Q_p(t,x_j) dt.
    \end{equation}
    It is the internal volume of the pipe network that is on the other
    side of $p$ than the inaccessible end $x_0$.
  \end{enumerate}
\end{algorithm}

\section{Numerical experiments}

All the programming here was done in GNU Octave \cite{octave}.

\begin{experiment}[Simple network with exact IRM] \label{exp1}
  We will start by solving for the trivial area of
  \cref{ex1}. Consider this a test for the implementation of the area
  reconstruction algorithm. We start by calculating the
  impulse-response matrix of vertices $A$ and $B$ while $C$ is
  inaccessible. $D$ is the junction.

  Set $A(x)=\SI{1}{\m\squared}$ everywhere, the gravitational
  acceleration $g=\SI{9.81}{\m\per\s\squared}$ and a wave speed of
  $a=\SI{1000}{\m\per\s}$.
  \begin{lstlisting}
    %% Physical parameters
    maxc = 1000; % maximal wave speed in the network
    g = 9.81; % standard gravity value (m/s^2)

    %% Network a-priori information
    V = 4; % total number of vertices
    L =   [400; 300; 1000]; % L(j) length of pipe j
    Adj = [1 4; 2 4; 4 3]; % pipe j goes from vertex Adj(j,1) to vertex Adj(j,2)
    A0 = [1,1]; % area at accessible pipe ends
    a0 = [maxc, maxc]; % wave speed at accessible pipe ends
  \end{lstlisting}

  Let us calculate the IRM by hand. If a flow of $V_0 \delta_0(t)$ is
  induced at point $A$, then it creates a propagating solution $Q =
  V_0 \delta(a x - t)$, $H= a V_0 \delta(ax-t) / (g A)$ along $AD$. If
  a pressure pulse of magnitude $M \delta_0$ is incident to $D$, then
  it transmits two pulses of magnitude $2M \delta_0 /3$ to $BD$ and
  $DC$, and reflects one pulse of magnitude $-M \delta_0 / 3$ back to
  $AD$. These follow from the junction conditions of
  \cref{scalarH,Kirchhoff}. Also, if a similar pressure pulse is
  incident to a boundary point with boundary condition $Q=0$, then a
  pulse of the same magnitude (no sign change) is reflected. However
  at that boundary point the pressure is measured as $2M
  \delta_0$. These considerations produce the following
  impulse-response matrix
  \begin{align*}
    K_{AA}(t) &= \frac{a}{g A} \Big( \delta_0\big(t\big) - \frac23
    \delta_0\big(t-\SI{0.8}{\s}\big) + \frac89 \delta_0\big(t -
    \SI{1.4}{\s}\big) + \frac29 \delta_0\big(t - \SI{1.6}{\s}\big) +
    \ldots \Big) \\ K_{BB}(t) &= \frac{a}{g A} \Big(
    \delta_0\big(t\big) - \frac23 \delta_0\big(t-\SI{0.6}{\s}\big) +
    \frac29 \delta_0\big(t-\SI{1.2}{\s}\big) + \frac89
    \delta_0\big(t-\SI{1.4}{\s}\big) + \ldots \Big) \\ K_{AB}(t) &=
    \frac{a}{g A} \Big( \frac43 \delta_0\big(t-\SI{0.7}{\s}\big) -
    \frac49 \delta_0\big(t-\SI{1.3}{\s}\big) - \frac49
    \delta_0\big(t-\SI{1.5}{\s}\big) + \ldots \Big) \\ K_{BA}(t) &=
    K_{AB}(t)
  \end{align*}
  for time $0\leq t \leq \SI{1.6}{\s}$.

  We must have $2\tau\leq\SI{1.6}{\s}$, and so the algorithm can solve
  for the area up to points of the network that are at most $a\tau =
  \SI{800}{\m}$ from each accessible end. Hence we can solve for the
  area only up to $\SI{400}{\m}$ in pipe $DC$ from the junction
  $D$. The reflections $k_{ij}$ used in \cref{alg1} are
  \begin{align*}
    k_{AA}(t) &= \frac{a}{g A} \Big( - \frac23
    \delta_0\big(t-\SI{0.8}{\s}\big) + \frac89 \delta_0\big(t -
    \SI{1.4}{\s}\big) + \frac29 \delta_0\big(t - \SI{1.6}{\s}\big) +
    \ldots \Big) \\ k_{BB}(t) &= \frac{a}{g A} \Big( - \frac23
    \delta_0\big(t-\SI{0.6}{\s}\big) + \frac29
    \delta_0\big(t-\SI{1.2}{\s}\big) + \frac89
    \delta_0\big(t-\SI{1.4}{\s}\big) + \ldots \Big) \\ k_{AB}(t) &=
    \frac{a}{g A} \Big( \frac43 \delta_0\big(t-\SI{0.7}{\s}\big) -
    \frac49 \delta_0\big(t-\SI{1.3}{\s}\big) - \frac49
    \delta_0\big(t-\SI{1.5}{\s}\big) + \ldots \Big) \\ k_{BA}(t) &=
    k_{AB}(t).
  \end{align*}
  \begin{lstlisting}[firstnumber=last]
    %% Measurements
    dt = 10/maxc; % in one time-step the wave propagates 10m
    experiment_duration = 1.61; % impulse-response matrix should go from time t=0 to experiment_duration.
    t = (0:dt:experiment_duration)';
    k = cell(2,2); % k is the response matrix
    % The below is calculated by hand for a network as above.
    k{1,1} = a0(1)/(A0(1)*g) *(...
    -2/3*1/dt.*(t >= 0.8-dt/2).*(t < 0.8+dt/2) ...
    +8/9*1/dt.*(t >= 1.4-dt/2).*(t < 1.4+dt/2) ...
    +2/9*1/dt.*(t >= 1.6-dt/2).*(t < 1.6+dt/2) ...
    );
    k{2,2} = a0(2)/(A0(2)*g) *(...
    -2/3*1/dt.*(t >= 0.6-dt/2).*(t < 0.6+dt/2)...
    +2/9*1/dt.*(t >= 1.2-dt/2).*(t < 1.2+dt/2) ...
    +8/9*1/dt.*(t >= 1.4-dt/2).*(t < 1.4+dt/2) ...
    );
    k{1,2} = a0(2)/(A0(2)*g) *(...
    +4/3*1/dt.*(t >= 0.7-dt/2).*(t < 0.7+dt/2)...
    -4/9*1/dt.*(t >= 1.3-dt/2).*(t < 1.3+dt/2) ...
    -4/9*1/dt.*(t >= 1.5-dt/2).*(t < 1.5+dt/2) ...
    );
    k{2,1} = k{1,2};
  \end{lstlisting}

  Let us define the action time functions for various points $p$ in
  the network. If $p \in AD$ then the action time is
  \[
  f^{AD}_p(x) = \begin{cases}
    TT(A,p), & x = A\\
    0, & x = B\\
    0, & x = C
  \end{cases}
  \]
  and for $p \in BD$
  \[
  f^{BD}_p(x) = \begin{cases}
    0, & x = A\\
    TT(B,p), & x = B\\
    0, & x = C
  \end{cases}.
  \]

  The travel-time from $A$ to $D$ is $\SI{0.4}{\s}$, and from $B$ to
  $D$ it is $\SI{0.3}{\s}$. Recall that the IRM has been measured only
  for time $t\leq \SI{1.6}{\s}$. Hence we can solve for the area up to
  $\SI{400}{\m}$ into pipe $DC$. Let the point $p\in DC$ be given by
  the action time function
  \[
  f^{DC}_p(x) = \begin{cases}
    \SI{0.4}{\s} + t_p, & x = A\\
    \SI{0.3}{\s} + t_p, & x = B\\
    0, & x = C
  \end{cases}
  \]
  where $0\leq t_p\leq \SI{400}{\m}/a = \SI{0.4}{\s}$ is the
  travel-time from $D$ to $p$ and recalling that we can solve the area
  only up to $\SI{400}{\m}$ from $D$.
  \begin{lstlisting}[firstnumber=last]
    %% Set parameters
    tau = 0.8; % area will be solved up to 2*tau*maxc from accessible point furthest to inaccessible point x3.
    assert(2*tau <= experiment_duration, 'Can recover area only up to experiment_duration*maxc/2.');
    reguparam = 1e-5; % Tikhonov regularization parameter
    dx = dt*maxc; % into how big chunks we discretize the pipe

    %% x-discretization
    maxL = [L(1); L(2); tau*maxc - max(L(1),L(2))];
    maxL = min(L,maxL); % maximum length of pipes that can be reached from all accessible vertices in time tau
    M = floor(maxL/dx); % number of discretized segments of length dx in each pipe

    %% Action times to various points in the network
    f = cell(length(L),1); % a different action time formula for points on each pipe

    % action times to pipe 1 (AD) points
    f{1} = nan(M(1),2);
    f{1}(:,1) = (1:M(1))'*dx./maxc;
    f{1}(:,2) = zeros(size((1:M(1))'));
    % action times to pipe 2 (BD) points
    f{2} = nan(M(2),2);
    f{2}(:,1) = zeros(size((1:M(2))'));
    f{2}(:,2) = (1:M(2))'*dx./maxc;
    % action times to pipe 3 (DC) points
    f{3} = nan(M(3),2);
    f{3}(:,1) = (L(1)+(1:M(3))'*dx)./maxc;
    f{3}(:,2) = (L(2)+(1:M(3))'*dx)./maxc;
  \end{lstlisting}

  Let us apply \cref{alg1} next. The numerical implementation of
  \cref{alg2}, \verb#makeHeq1#, is in the appendix.
  \begin{lstlisting}[firstnumber=last]
    %% Solve for the cross-sectional area
    % Initialize cells for saving various vectors
    pipeVolume = cell(length(L),1); % volume of network cut by p
    pipeArea = cell(length(L),1); % cross-sectional area at p
    pipeX = cell(length(L),1); % x-coordinates of points p

    for P=1:length(L) % P indexes the pipe number (AD, BD, DC)
      assert(max(max(f{P})) - tau <= dt/4, 'Travel-times must be at most tau.');
      V = nan(M(P),1); % volume of the network up to points p
      for p=1:M(P) % p indexes the point $p$ inside pipe P
        % use Algorithm 2:
        Qtau = makeHeq1(t, k, tau, f{P}(p,:), a0, g, A0, reguparam);
        V(p) = 0;
        % add to V the volume of water from each accessible end that would have gone INTO the pipe to make H=1 at t=tau:
        for ii = 1:length(Qtau)
          V(p) = sum(Qtau{ii})*dt + V(p)
        end
      end
      pipeVolume{P} = maxc^2/g*V;
      pipeArea{P} = (pipeVolume{P}(2:end)-pipeVolume{P}(1:end-1))/dx;
      pipeX{P} = (1:M(P)-1)'.*dx;
    end
  \end{lstlisting}

  A plot of the cross-sectional areas is shown in \cref{fig_ex1}.
  \begin{figure}
    \centering
    \includegraphics[width=\textwidth]{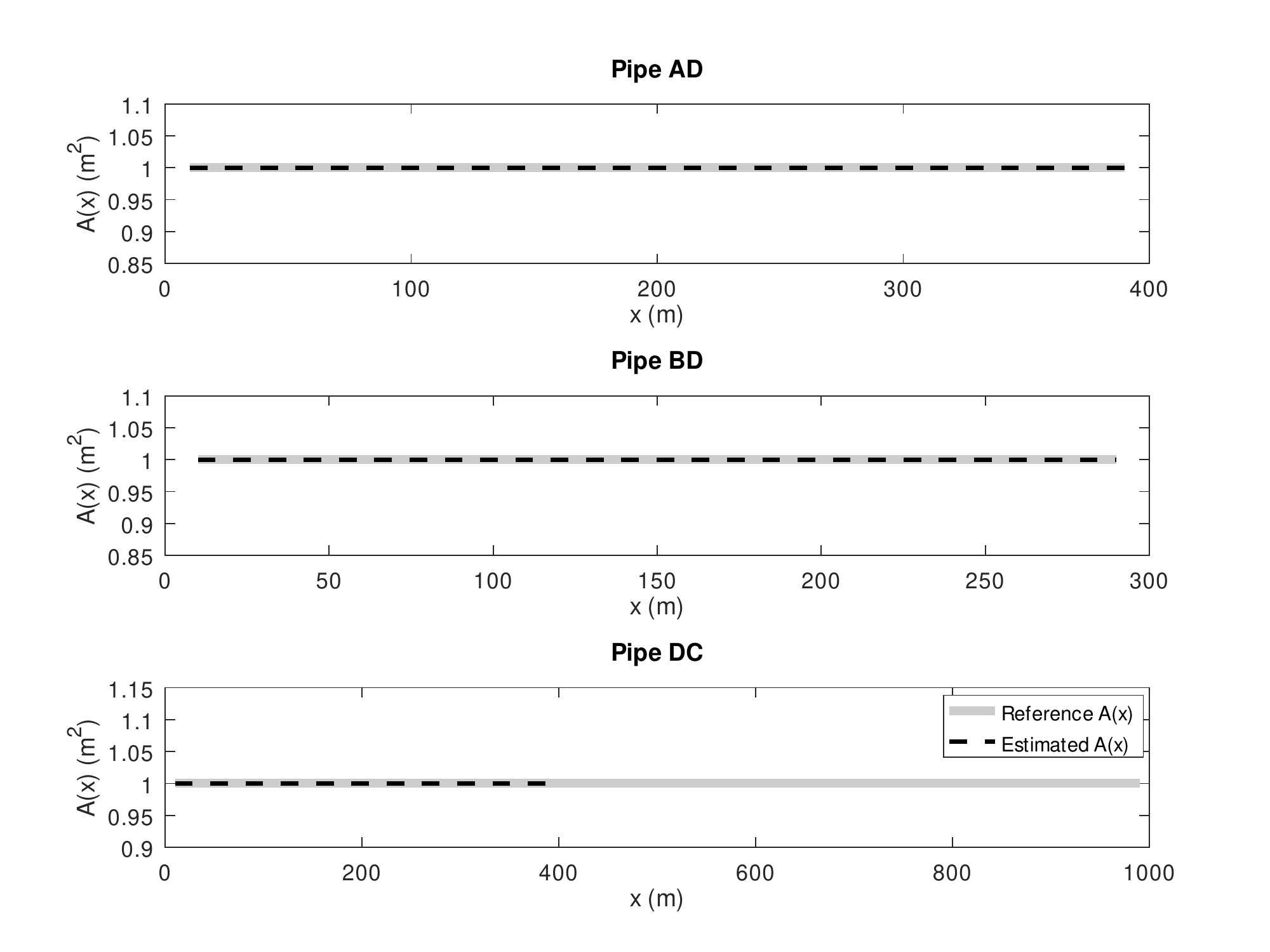}
    \caption{Solved cross-sectional areas of \cref{exp1}}
    \label{fig_ex1}
  \end{figure}
\end{experiment}

\begin{experiment}[Simulated IRM] \label{exp2}
  In the second experiment we consider a star-shaped network with four
  leaves, ending in points $A,B,C,D$. The internal node is denoted
  $E$. Let $D$ be the inaccessible end. Take the following lengths
  \begin{align*}
    AE = ({0,\SI{300}{\m}})&, &BE = ({0,\SI{400}{\m}})&, &CE =
    ({0,\SI{400}{\m}})&, &ED = ({0,\SI{500}{\m}}),&
  \end{align*}
  as presented in \cref{exp2_fig}.
  \begin{figure}
    \centering
    \includegraphics{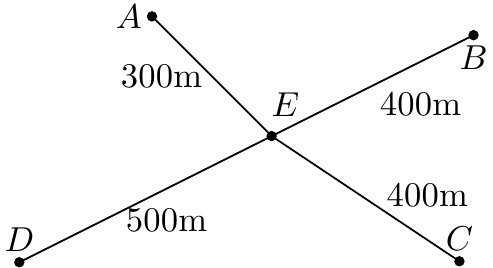}
    \caption{A more complicated network}
    \label{exp2_fig}
  \end{figure}

  We set a constant wave speed of $a=\SI{1000}{\m\per\s}$ everywhere
  and an area function that models blockages at certain locations.
  \begin{lstlisting}
    %% physical parameters
    maxc = 1000; % max wave speed
    g = 9.81; % standard gravity value (m/s^2)
  \end{lstlisting}

  \begin{lstlisting}[firstnumber=last]
    %% Network
    V = 5; % total number of vertices
    L =   [300; 400; 400; 500]; % pipe lengths
    Adj = [1 5; 2 5; 3 5; 5 4]; % to avoid display problems make arrows point towards inaccessible node (nbr 4)
    Afunc1 = @(s)( ones(size(s)) );
    Afunc2 = @(s)( 2*ones(size(s)) - (s>350).*(s<375).*0.6);
    Afunc3 = @(s)( ones(size(s)) - (s>210).*(s<250).*0.2);
    Afunc4 = @(s)( ones(size(s)) - (s>410).*(s<450).*0.4 - (s>150).*(s<250).*0.2);

    afunc1 = @(s)( maxc*ones(size(s)) );
    afunc2 = @(s)( maxc*ones(size(s)) );
    afunc3 = @(s)( maxc*ones(size(s)) );
    afunc4 = @(s)( maxc*ones(size(s)) );

    Afunc = {Afunc1; Afunc2; Afunc3; Afunc4};
    afunc = {afunc1; afunc2; afunc3; afunc4};

    BVtype = [1; 1; 1; 1; NaN]; % Are inputs pressure (0) or flow (1)
  \end{lstlisting}

  We simulate the IRM by an a finite-difference time domain (FDTD)
  algorithm with the following caveats: we use a Courant number
  smaller than one, simulate using a high resolution, and then
  interpolate the IRM to a lower time-resolution and use this as input
  for the inversion algorithm. This is to avoid the inverse crime
  \cite{KaipioSomersalo} which makes inversion algorithms give
  unrealistically good results when applied on data simulated with the
  same resolution or model as the inversion algorithm uses.
  \begin{lstlisting}[firstnumber=last]
    %% Measurements
    % FDTD parameters
    dx = 5;
    courant = 0.95;
    dt = courant*dx/maxc;
    experiment_duration = 1.9;
    doPlot = 0; % do we want to observe the FDTD simulation

    % Boundary area and wave speed. Accessible ends are 1, 2 and 3
    A0 = [Afunc1(0); Afunc2(0); Afunc3(0)];
    a0 = [afunc1(0); afunc2(0); afunc3(0)];
  \end{lstlisting}

  For numerical reasons instead of sending a unit impulse $Q = \nu
  \delta$ we send a unit step-function, and then differentiate the
  measurements with respect to time. We will not show the
  implementation of \verb#FDTD# and the self-explanatory functions
  \verb#removeInitialPulse#, \verb#medianSmooth# and
  \verb#differentiate# because they are not the focus of this already
  rather long article. The main point is the inversion algorithm.
  \begin{lstlisting}[firstnumber=last]
    t = (0:dt:experiment_duration)';
    F = bndrySourceOn(t);
    k = cell(3,3);
    for ii=1:3
      % Create input flows: constant flow at ii, otherwise zero
      Vdata = cell(V,1);
      for vv = 1:V
        Vdata{vv} = zeros(size(F));
      end
      Vdata{ii} = F;

      % Simulate the measurements with the given F
      [Hhist, Qhist, thist] = FDTD(L, V, Adj, BVtype, Vdata, t, ...
        doPlot, maxc, g, dx, courant, Afunc, afunc);
      % Remove the initial pulses and differentiate with respect to time
      for jj=1:3
        H = Hhist{jj};
        if(ii==jj)
          H = removeInitialPulse(H, thist, a0(ii), g, A0(ii));
        end
        % Smoothen slightly to make differentiation well behaved  
        H = medianSmooth(H, floor(0.02/(thist(2)-thist(1))));
        H = differentiate(H, thist);
        k{ii,jj} = H;
      end
    end

    % Avoid the inverse crime by interpolation to a lower resolution
    dx = 7;
    dt = dx / maxc;
    t = (thist(1): dt : experiment_duration)';
    for ii=1:size(k,1)
      for jj=1:size(k,2)
        k{ii,jj} = interp1(thist, k{ii,jj}, t);
      end
    end
  \end{lstlisting}
  The simulation gives us the following response matrix, as shown in
  \cref{exp2_input_fig}.
  \begin{figure}
    \centering
    \includegraphics[width=\textwidth]{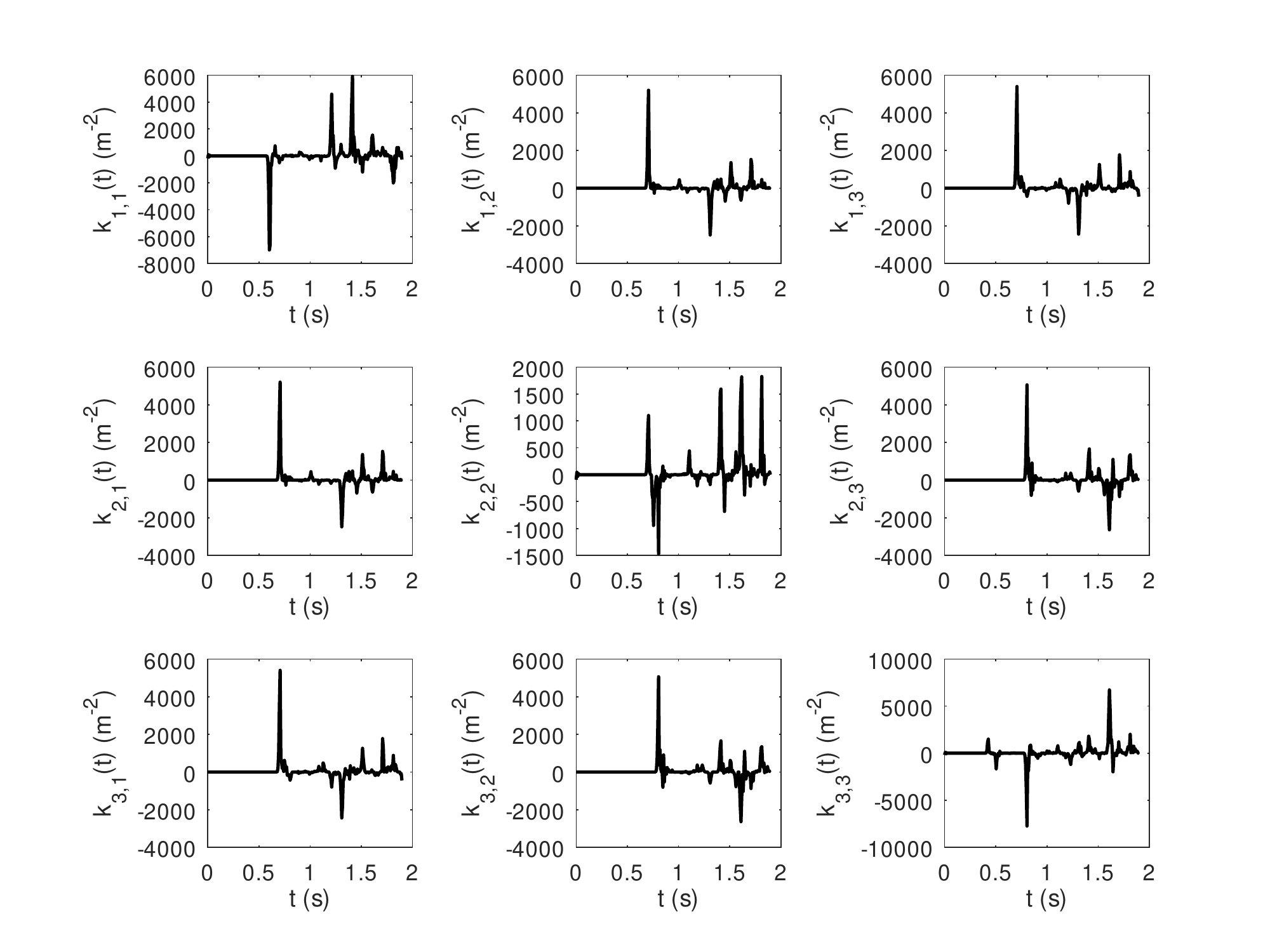}
    \caption{IRM for \cref{exp2}}
    \label{exp2_input_fig}
  \end{figure}

  Now solving the inverse problem has the same logic as in
  \cref{exp1}. There are two differences, a large one and a small
  one. The former is that the action time functions are of course
  different. This is what encodes the network topology for the
  inversion algorithm. And secondly we must use quite a lot of
  regularization when solving for the area of $ED$. This is because of
  the ``numerical error'' introduced on purpose by the Courant number
  smaller than one and interpolating the measurements to avoid doing
  an inverse cime.
  \begin{lstlisting}[firstnumber=last]
    %% Set parameters
    reguparam = [1e-5, 1e-5, 1e-5, 1e0]; % regularization parameter to Tikhonov regularization in various pipes
    tau = 0.9;

    %% x-discretization
    maxL = [L(1); L(2); L(3); tau*maxc - max([L(1) L(2) L(3)])];
    maxL = min(L,maxL); % maximum length of pipes that can be reached from all accessible vertices in time tau
    M = floor(maxL/dx); % Number of discretized segments of length dx in each pipe
    f = cell(length(L),1);

    %% Action times to various points in the network
    % action times to pipe 1 (AE) points
    f{1} = nan(M(1),3);
    f{1}(:,1) = (1:M(1))'*dx./maxc;
    f{1}(:,2) = zeros(size((1:M(1))'));
    f{1}(:,3) = zeros(size((1:M(1))'));
    % action times to pipe 2 (BE) points
    f{2} = nan(M(2),3);
    f{2}(:,1) = zeros(size((1:M(2))'));
    f{2}(:,2) = (1:M(2))'*dx./maxc;
    f{2}(:,3) = zeros(size((1:M(2))'));
    % action times to pipe 3 (CE) points
    f{3} = nan(M(3),3);
    f{3}(:,1) = zeros(size((1:M(3))'));
    f{3}(:,2) = zeros(size((1:M(3))'));
    f{3}(:,3) = (1:M(3))'*dx./maxc;
    % action times to pipe 4 (ED) points
    f{4} = nan(M(4),3);
    f{4}(:,1) = (L(1)+(1:M(4))'*dx)./maxc;
    f{4}(:,2) = (L(2)+(1:M(4))'*dx)./maxc;
    f{4}(:,3) = (L(3)+(1:M(4))'*dx)./maxc;
  \end{lstlisting}
  Then applying \cref{alg1} gives the solution as before.
  \begin{lstlisting}[firstnumber=last]
    %% Solve for the cross-sectional area
    % Initialize cells for saving various vectors
    pipeVolume = cell(length(L),1); % volume of network cut by p
    pipeArea = cell(length(L),1); % cross-sectional area at p
    pipeX = cell(length(L),1); % x-coordinate of points p

    for P=1:length(L) % P indexes the pipe number (AE, BE, CE, ED)
      assert(max(max(f{P})) - tau <= dt/4, 'Travel-times must be at most tau.');
      V = nan(M(P),1); % volume of the network up to points p
      for p=1:M(P) % p indexes the point $p$ inside pipe P
        Qtau = makeHeq1(t, k, tau, f{P}(p,:), a0,g,A0, reguparam(P));
        V(p) = 0;
        % add to V the total volume of water from accessible end that would have gone INTO the pipe to make H=1 at t=tau.
        for ii = 1:length(Qtau)
          V(p) = sum(Qtau{ii})*dt + V(p);
        end
      end
      pipeVolume{P} = maxc^2/g*V;
      pipeArea{P} = (pipeVolume{P}(2:end)-pipeVolume{P}(1:end-1))/dx;
      pipeX{P} = (1:M(P)-1)'.*dx;
    end
  \end{lstlisting}
  The solution is displayed in \cref{fig_ex2}. The gray uniform line
  represents the original cross-sectional area. The dashed line is the
  solution to the inverse problem.
  \begin{figure}
    \centering
    \includegraphics[width=\textwidth]{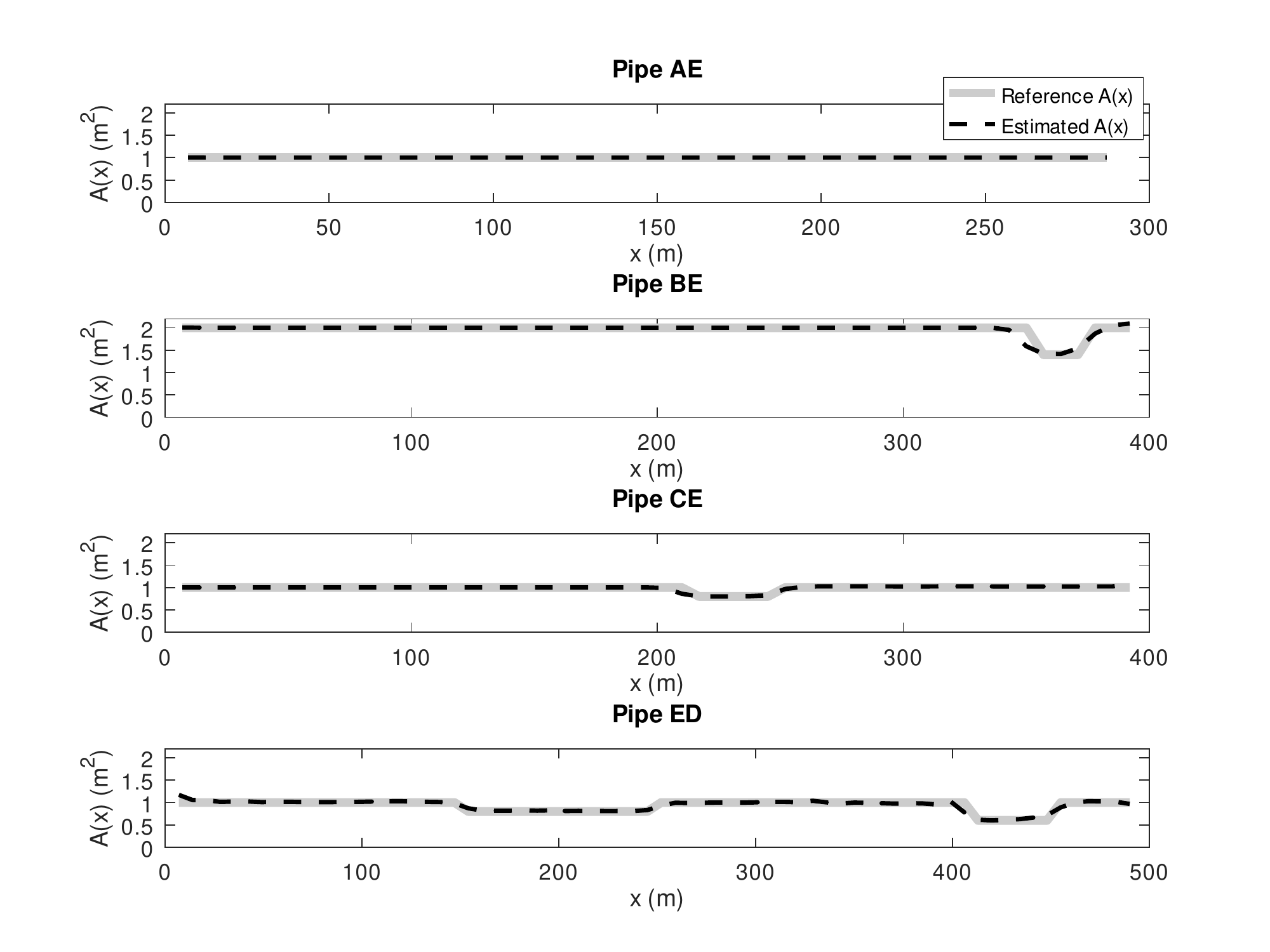}
    \caption{Solved cross-sectional areas of \cref{exp2}}
    \label{fig_ex2}
  \end{figure}
\end{experiment}

\section{Discussion and conclusions} \label{discussion}
We have developed and implemented an algorithm that reconstructs the
internal cross-sectional area of pipes, filled with fluid, in a
network arrangement. The region where the area is reconstructed must
form a tree network and the input to our algorithm is the
impulse-response measurements on all of the tree's ends except for
possibly one. The algorithm involves solving a boundary integral
equation which is mathematically solvable in the case of perfect
measurements and model (see the appendix). We wrote a step-by-step
reconstruction algorithm and tested it using two numerical
examples. The first one has a perfectly discretized measurement, and
the second one has a discretization and numerical error introduced on
purpose. Even with these errors, which are very typical when doing
real-world measurements, Tikhonov regularization allows us to
reconstruct the internal cross-sectional area with good
precision. However this is just a first study of this algorithm, and a
more in-depth investigation would be required for a more complete
picture of its ability in the case of noise and other errors in the
data. The theory is based on our earlier work \cite{area1} on solving
for the area of one pipe, and on an iterative time-reversal boundary
control algorithm \cite{Oksanen} in the context of multidimensional
manifolds.

We assumed in several places that the wave speed $a(x)$ is a known
constant. What if it is not? We considered this situation for a single
pipe in our previous article, \cite{area1}. In there, if the area is
known and the speed is unknown, then the algorithm can be slightly
modified to determine the wave speed profile along the pipe. If both
the wave speed and the area are unknown then it can determine the
hydraulic impedance $Z = a/gA$ as a function of the travel-time
coordinate. However this is not so straightforward for the network
case because of the more complicated geometry. What one should do
first is find an algorithm to solve this problem: the wave speed is
constant and known, the area is unknown, the topology of the network
is known, but the pipe lengths are unknown. We leave this problem for
future considerations as this article is already quite long. But it is
an important question, because in fact in some applications the
anomaly is the loss in pipe thickness rather than the change in pipe
area and this is often revealed by finding the wave speed along the
pipes assuming that the area remains unchanged \cite{Gong2014}.

In \cref{exp1} we did not use regularization and the solution was
perfect, as shown in \cref{fig_ex1}. However here we had the simplest
tree network, the simplest area formula, and perfectly discretized
measurements.

In \cref{exp2} we simulated the measurements using a finite difference
time domain (FDTD) method with Courant number smaller than one. Using
Tikhonov regularization was essential, and produced the reconstruction
shown in \cref{fig_ex2}. Without regularization one would get a large
artifact, as in \cref{fig_artifact}.
\begin{figure}
  \centering
  \includegraphics[width=\textwidth]{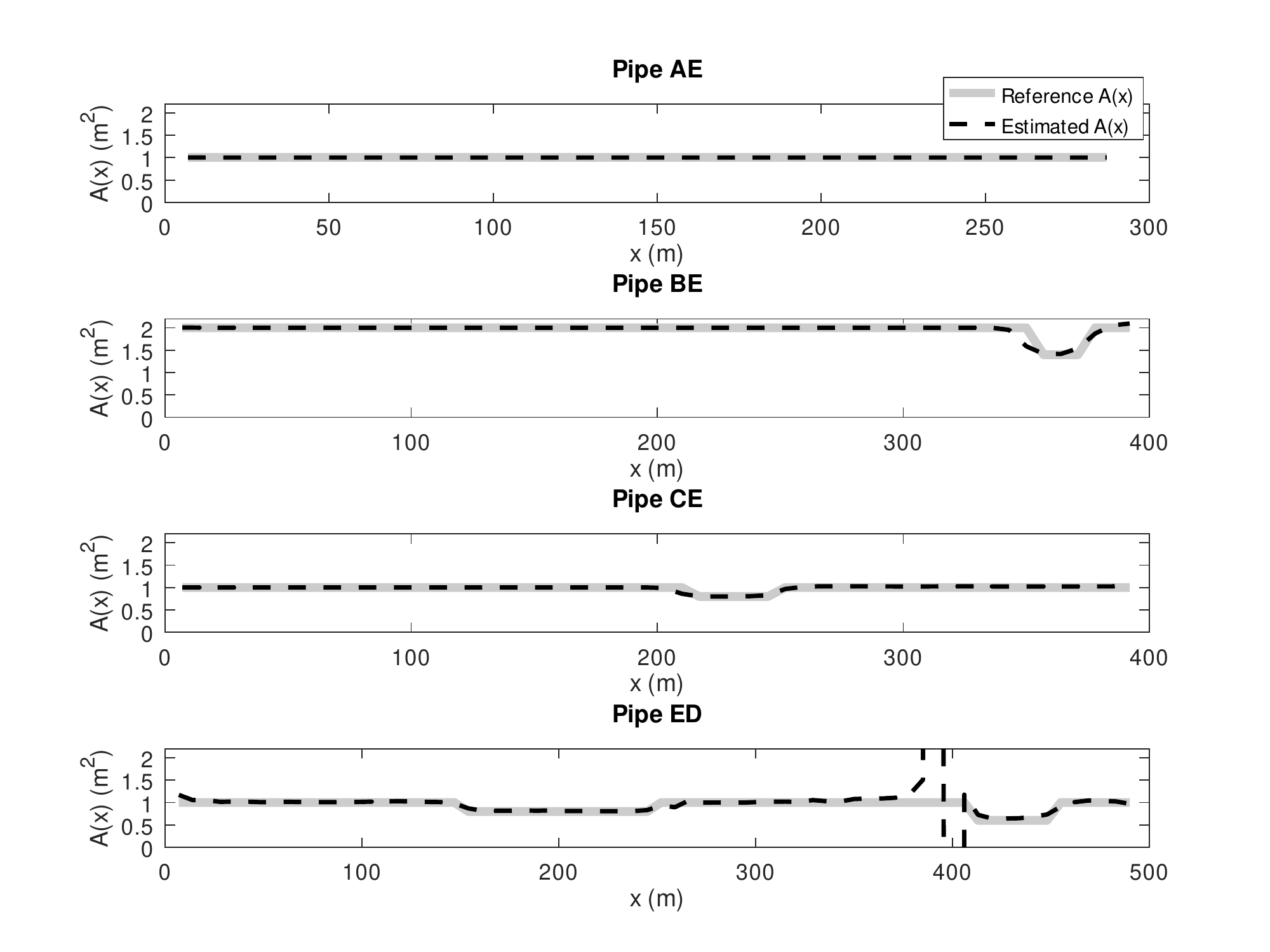}
  \caption{Solving without regularization}
  \label{fig_artifact}
\end{figure}
The reasons for using imperfect measurements are to demonstrate the
algorithm's stability. We could have calculated the impulse-response
matrix using the more exact method of characteristics. However in
reality, when dealing with data from actual measurement sensors, one
would never get such perfect inputs to the inversion algorithm. First
of all there are modelling errors, and secondly, measurement noise. By
using FDTD with a small Courant number and also by inputting
measurements with a rougher discretization than in the direct model we
purposefully seek to avoid inverse crimes, as decribed in Section~1.2
of \cite{KaipioSomersalo}.

An inverse crime happens when the measurement data is simulated with
the same model as the inversion algorithm assumes. Typically in these
cases the numerical reconstruction looks unrealistically good, even
with added measurement noise, and therefore does not reflect the
method's actual performance in real life, where models are always
approximations.

Several directions of investigation still remain open. A more in-depth
numerical study should be concluded, with proper statistical analysis
and for example finding the signal to noise ratio that still gives
meaningful reconstructions. There is also the issue of measuring the
impulse-response matrix. It would be very much appreciated if there
was no need to actually close almost all the valves in the network to
perform measurements. Hence one could investigate various other types
of boundary measurements and see how to process them to reveal the
impulse-response matrix used here. To make even further savings in
assessing the water supply network's condition, one could also try
implementing the theory from \cite{Kurasov--tree} as an algorithm.
Their theoretical result only requires the so-called ``backscattering
data'' from the network: instead of having to measure the full
impulse-response matrix $(K_{ij})_{i,j=1}^N$ it would be enough to
measure its diagonal $(K_{ii})_{i=1}^N$. This means that the pressure
needs to be measured only at the same pipe end as the flow is
injected. So the exact same type of measurement as is done for a
single pipe, as in \cite{area1}, should be done at each accessible
network end. For a network with $N+1$ ends this translates into $N$
measurements compared to the $N^2$ for the full IRM.

Lastly we comment on our algorithm. What is interesting about it, is
that there is no need to solve the area in the whole network at the
same time. To save on computational costs one could for example
reconstruct the cross-sectional area only in a small region of
interest in a single pipe even when the network is otherwise quite
large. Alternatively one could parallelize the process and solve for
the whole network very fast using multiple computational cores. These
would likely be impossible when having only the backscattering data
described above.

The steps in \cref{alg1,alg2} describe the reconstruction process in
detail. In essence only a single matrix inversion (or least squares or
other minimization) is required for solving the area at any given
point in the network. The size of this matrix depends on how finely
the impulse-response matrix has been discretized, and how far from the
boundary this point is. The rougher the discretization, the faster
this step. However having a poor discretization leads to a larger
numerical error, and one might then need to use a regularization
scheme. All the numerical examples in this article were performed with
an office laptop from 2008, and they took a few minutes to calculate
including simulating the measurements. The algorithm is
computationally light, simple to implement, and thus suitable for
practical applications.

\section{Acknowledgements}
This research was partly funded by the fullowing grants.
\begin{itemize}
\item T21-602/15R Smart Urban water supply systems(Smart UWSS), and
\item 16203417 Blockages in water pipes: theoretical and experimental
  study of wave-blockage interaction and detection.
\end{itemize}
Furthermore all authors declare that there are no conflicts of
interest in this paper.

\section{Appendix}
\subsection{Existence of a solution} We give a list of the steps
involved in proving that the equation in \cref{treeEquationForH1} has
a solution when the impulse-response matrix has been measured exactly
and there is no modelling errors. Since the ultimate goal for our work
is on assessing the quality of water supply network pipes, in the end
there will be both modelling errors and measurement noise, so we will
not give a complete formal proof. However we give enough details so
that any mathematician specialized in the wave equation can fill the
gaps after choosing suitable function space classes and assumptions
for the various objects.
\begin{enumerate}
\item\label{a} Given $p\in\mathbb G \setminus \mathbb J$ which defines
  the admissible set $D_p$ and action time $f$ according to
  \cref{admissibleSet}, and with $\tau > \max f$ and $h_0>0$, we want
  to solve
  \begin{align*}
    h_0 &= \frac{a}{A(x_j)}\nu(x_j) Q_p(t,x_j) \\ &\qquad + \sum_{x_i
      \in \partial\mathbb G \setminus \mathbb J} \frac{\nu(x_i)}{2}
    \int_0^\tau Q_p(s,x_i) \big( k_{ij}(\abs{t-s}) +
    k_{ij}(2\tau-t-s)\big) ds
  \end{align*}
  when $\tau-f(x_j) < t \leq \tau$, $x_j \in \partial \mathbb G
  \setminus \{x_0\}$, and $Q_p(s,x_j) = 0$ when $t \leq \tau -
  f(x_j)$, $x_j\in\mathbb G \setminus \{x_0\}$.
\item\label{b} We extend boundary data from the interval $(0,\tau)$ to
  $(-\infty,+\infty)$ and also absorb the interior boundary normal, by
  writing
  \[
  \mathscr Q_p(t,x_j) = \begin{cases} \nu(x_j) Q_p(t,x_j), &0\leq
    t\leq\tau,\\ \nu(x_j) Q_p(2\tau-t,x_j), & \tau < t \leq 2\tau,
    \\ 0, &\text{otherwise}. \end{cases}
  \]
\item\label{c} Solving
  \begin{align*}
    \tilde h_0(t,x_j) &= \frac{a}{A(x_j)} \mathscr Q_p(t,x_j)
    \\ &\qquad+ \sum_{x_i \in \partial\mathbb G \setminus \mathbb
      J} \frac{\chi(t,x_j)}{2} \int_0^{2\tau} \mathscr Q_p(s,x_i)
    k_{ij}(\abs{t-s}) ds
  \end{align*}
  in
  \begin{align*}
    L^2_0 &= \Big\{ F \in L^2(\mathbb R \times \partial\mathbb G
    \setminus \{x_0\}) \,\Big|\, \\&\qquad F(t,x_j) = 0 \text{ if }
    \abs{t-\tau}\geq f(x_j),\, F(2\tau-t,x_j) = F(t,x_j) \Big\}
  \end{align*}
  where
  \[
  \chi(t,x_j) = \begin{cases} 1, & \abs{t-\tau}<f(x_j),\\ 0, &
    \abs{t-\tau}\geq f(x_j), \end{cases} \qquad \tilde h_0(t,x_j) =
  h_0 \chi(t,x_j)
  \]
  is equivalent to solving the equation of \cref{a} in the
  corresponding $L^2$-based space. We equip $L^2_0$ with the inner
  product
  \[
  \langle A, B \rangle = \sum_{x_j\in\partial\mathbb G \setminus
    \mathbb J} \int_0^{2\tau} A(t,x_j) \overline{B(t,x_j)} dt
  \]
  where the complex conjugation can be ignored because all of our
  numbers are real.
\item\label{d} The equation in \cref{c} can be written as $\tilde h_0
  = \mathscr K \mathscr Q_p$ in $L^2_0$. Here we define the operator
  $\mathscr K$ by
  \begin{align*}
    \mathscr K F (t,x_j) &= \frac{a}{A(x_j)} F(t,x_j) \\&\qquad +
    \sum_{x_i \in \partial\mathbb G \setminus \mathbb J}
    \frac{\chi(t,x_j)}{2} \int_0^{2\tau} F(s,x_i) k_{ij}(\abs{t-s}) ds
  \end{align*}
  and it is a well-defined operator if $k_{ij}$ is a distribution of
  order $0$ which is the case if the cross-sectional area function $A$
  is piecewise smooth for example. It indeed maps $L^2_0 \to L^2_0$
  satisfying the support and time-symmetry conditions.
\item\label{e} We will show that $\mathscr K$ is a Fredholm operator
  $L^2_0\to L^2_0$ that is positive semidefinite.
\item\label{f} Recall the travel-time function $TT$, the action time
  function $f$ and the admissible set $D_p$ from \cref{a}. Define
  \begin{align*}
    \Omega &= \Big\{(t,x) \in \mathbb R \times \mathbb G \,\Big|\,
    \\&\qquad TT(x,x_j) + \abs{t-\tau} < f(x_j) \text{ for some }
    x_j \in \partial\mathbb G \setminus \{x_0\} \Big\}, \\ t_\pm(x) &=
    \tau \pm \max_{x_j\in\partial\mathbb G\setminus\{x_0\}} \big(
    f(x_j) - TT(x,x_j) \big).
  \end{align*}
  The set $\Omega$ coincides with the set where unique continuation
  from the boundary holds, as in \cref{treeUniqCont}. One can show
  that
  \[
  \Omega = \{(t,x)\in \mathbb R \times \mathbb G \mid x\in D_p, t_-(x)
  < t < t_+(x) \}.
  \]
\item\label{g} Let $F\in L^2_0$ be given and fixed. Let $H,Q:\mathbb R
  \times \mathbb G \to \mathbb R$ satisfy the network wave model of
  \cref{model} with boundary flow $F$.
\item\label{h} We see that $H(t,x) = Q(t,x) = 0$ when $x\in D_p$ and
  $t\leq t_-(x)$. This follows from the zero initial conditions,
  finite speed of wave propagation and the definition of $t_\pm$ and
  the action time function $f$.
\item\label{i} Define
  \[
  S = \frac{\sqrt{g A}}{a} H + \frac{\mu}{\sqrt{g A}} Q
  \]
  where $\mu(x)=+1$ if the positive direction of the coordinates at
  $x$ points towards $p$, and $\mu(x)=-1$ otherwise. Then one sees
  that
  \[
  \big( \partial_x - \tfrac{\mu}{a} \partial_t \big)(QH) = -
  \frac{1}{2} \partial_t (S^2)
  \]
  in $t\in\mathbb R$, $x\in \mathbb G \setminus \mathbb J$.
\item\label{j} By \cref{f,h,i} we see that
  \begin{align*}
    -\frac{1}{2} \int_\Omega \partial_t (S^2) dt dx &= -\frac{1}{2}
    \int_{D_p} \big(S^2(t_+(x),x) - S^2(t_-(x),x)\big) dx \\&=
    -\frac{1}{2} \int_{D_p} S^2(t_+(x),x) dx \leq 0.
  \end{align*}
\item\label{k} By \cref{i,j} and the divergence theorem we have
  \begin{align*}
    0 &\geq \int_\Omega \big( \partial_x - \tfrac{\mu}{a} \partial_t
    \big)(QH) dt dx = \int_\Omega \nabla_{t,x} \cdot \big(
    -\tfrac{\mu}{a} QH, QH \big) dt dx \\&= \int_{\partial\Omega}
    \nu_{t,x} \cdot \big(-\tfrac{\mu}{a},1\big) QH d\sigma(t,x)
  \end{align*}
  where $\nu_{t,x}$ is the external unit normal vector to $\Omega$ at
  $(t,x) \in \partial\Omega$.
\item\label{l} Let us split $\partial\Omega$ next. By \cref{f} we see
  that
  \begin{align*}
    \partial\Omega &= \{ (t,x_j) \mid x_j\in\partial\mathbb
    G\setminus\{x_0\}, \abs{t-\tau}\leq f(x_j) \} \\ &\qquad \cup
    \{(t,x) \mid x\in D_p, t=t_+(x) \} \cup \{(t,x) \mid x\in D_p,
    t=t_-(x) \} \\ & \qquad \cup \{(\tau,p)\}.
  \end{align*}
\item\label{m} On the first set in \cref{l} the external unit normal
  $\nu_{t,x_j}$ is given by $\nu_{t,x_j} = (0,-\nu(x_j))$ because
  $\nu$ was defined as the internal normal at the pipe ends.
\item\label{n} Consider the second set in \cref{l}. On each individual
  pipe segment $P \subset D_p$ the map $x \mapsto t_+(x)$ is linear
  (affine). How does $t_+$ change when we go from $x$ to $x+\Delta x$?
  Recall that $\abs{\Delta x} = a \abs{\Delta t}$, and since
  $(t_+(x),x)$ follows the characteristics, so $t_+(x+\Delta x) =
  t_+(x) \pm \Delta x / a$. Its value decreases if $\Delta x > 0$ and
  the positive direction of the coordinates $x$ point towards
  $p$. Both of these follow from \cref{f} and the definition of the
  action time function $f$. By the definition of $\mu$ in \cref{i} we
  see that
  \[
  \frac{\Delta x}{t_+(x+\Delta x) - t_+(x)} = - \frac{a}{\mu}.
  \]
  Thus the normal has a slope of $\mu/a$ and so
  \[
  \nu_{t,x} = \frac{(a,\mu)}{\sqrt{a^2+\mu^2}} =
  \frac{(a,\mu)}{\sqrt{1+a^2}}.
  \]
\item\label{o} By \cref{g,h,j,k,l,m,n} we get
  \begin{align*}
    0 &\geq -\frac{1}{2} \int_{D_p} S^2(t_+(x),x) dx =
    \int_{\partial\Omega} \nu_{t,x} \cdot \big(-\tfrac{\mu}{a},1\big)
    QH d\sigma(t,x) \\ &= \sum_{x_j\in\partial\mathbb G\setminus
      \mathbb J} \int_{\tau - f(x_j)}^{\tau+f(x_j)} -F(t,x_j) H(t,x_j)
    dt \\ &\qquad + \int_{D_p} \frac{(a,\mu)}{\sqrt{1+a^2}} \cdot
    \big( -\tfrac{\mu}{a}, 1 \big) (QH)(t_+(x),x) dx \\&\qquad +
    \int_{D_p} \nu_{t,x} \cdot \big(-\tfrac{\mu}{a}, 1\big)
    (QH)(t_-(x),x) dx + 0
  \end{align*}
  which gives
  \[
  \sum_{x_j\in\partial\mathbb G\setminus \mathbb J} \int_0^{2\tau}
  F(t,x_j) H(t,x_j) dt = \frac{1}{2} \int_{D_p} S^2(t_+(x),x) dx \geq
  0
  \]
  by the support condition of $F\in L^2_0$ given in \cref{c}. In
  detail we see that the $D_p$-integrals vanish because of the
  following. Firstly we see that $(a,\mu) \cdot (-\mu/a,1) = 0$ making
  the first integral over $D_p$ vanish. Secondly, by \cref{h},
  $(QH)(t_-(x),x)=0$ for $x\in D_p$. The integral over the singleton
  $(t,s) = (\tau, p)$ is zero because $QH$ is a function since $F$ is
  too.
\item\label{p} Recall \cref{HfromIRM} and the splitting of the IRM $K$
  into the impulse and response matrices in
  \cref{treeEquationForH1}. Thus
  \[
  H(t,x_j) = \frac{a}{A(x_j)g} F(t,x_j) + \sum_{x_i\in\partial\mathbb
    G\setminus \{x_0\}} \int_0^t F(s,x_i) k_{ij}(t-s) ds.
  \]
\item\label{q} Combining \cref{o,p} and changing the order of
  integration in double integrals, we arrive at
  \[
  \langle \mathscr K F, F \rangle = \frac{1}{2} \int_{D_p}
  S^2(t_+(x),x) dx \geq 0
  \]
  for arbitrary $F\in L^2_0$ and where $S,H,Q$ are defined by
  \cref{g,i}.
\item\label{r} We cannot show that $\mathscr K F = 0$ implies $F=0$
  unless the network is a single pipe. For a counter example choose a
  three-pipe network with constant unit area and wave speed, and each
  pipe is of unit length. Then send a pulse from one end and the same
  with opposite sign from another end. These pulses meet at the
  junction at the exact same time and cancel out the pressure there,
  so nothing propagates to the third pipe. Then they continue on the
  original two pipes until they hit the boundaries. The boundary flow
  $F$ should be then chosen so that these pulses are absorbed
  completely instead of reflecting back. This boundary flow $F$ gives
  $\langle \mathscr K F, F \rangle = 0$ by \cref{q}, but $F\neq0$.
\item\label{s} The existence of a solution follows from the following:
  that $\mathscr K$ is self-adjoint and Fredholm, and that $\tilde h_0
  \perp \ker \mathscr K $, so then there is $F$ giving $\mathscr K F =
  \tilde h_0$. We shall show these next.
\item\label{t} If we assume that $A(x)$ is smooth enough, then the
  components of the response matrix $(k_{ij})_{i,j=1}^N$ have at most
  a finite number of delta-functions arising from the junctions of
  $\mathbb G$ (in the time-interval ${({0,2\tau})}$), and are
  otherwise a function. Hence $\mathscr K$ can be split into a
  multiplication by nonvanishing functions, a finite number of
  translations that are also multiplied by such functions, and a
  smoothing integral operator which is compact $L^2_0\to L^2_0$. The
  identity and translations multiplied by non-vanishing functions are
  Fredholm operators, and the rest is compact. Hence the whole
  $\mathscr K$ is Fredholm. The space $L^2_0$ is Hilbert, and
  $\mathscr K$ is easily seen to be symmetric there because $k_{ij} =
  k_{ji}$. It is a bounded operator. Hence it is self-adjoint. Thus
  $\mathscr K$ is a self-adjoint Fredholm operator.
\item\label{u} Let $w(s) = \int_{\mathbb G} H(s,x) \frac{g A}{a^2} dx$
  where $H,Q$ are given by \cref{g}. By differentiating with respect
  to $s$, using \cref{WH1}, calculating the $x$-integral, and
  integrating with respect to $s$ we see that
  \[
  \int_{D_p} H(\tau,x) \frac{g A}{a^2} dx = w(\tau) =
  \sum_{x_j\in\partial\mathbb G \setminus\{x_0\}} \int_0^\tau \nu(x_j)
  Q(s,x_j) ds
  \]
  because $Q=0$ near the boundary point $x_0$ up to time $\tau$. This
  is the same as in \cref{impedanceByBoundaryData}.
\item\label{v} Let $F\in L^2_0$ with $\mathscr K F = 0$ and $H,Q$ as
  in \cref{g}. Set $h_0=0$ in \cref{treeEquationForH1} to conclude
  that $H(\tau,x)=0$ for $x\in\mathbb G$. Recall that the equations in
  that theorem are formulated using $\mathscr K$ here (see \cref{c}).
\item\label{w} The time-symmetry of $F\in L^2_0$ gives the first
  equality below. \cref{u,v} give the second and third equality,
  respectively. So
  \[
  \langle \tilde h_0, F \rangle = \sum_{x_j\in\partial\mathbb G
    \setminus\{x_0\}} 2 h_0 \int_0^\tau F(s,x_j) ds = 2 h_0 \int_{D_p}
  H(\tau,x) \frac{g A}{a^2} dx = 0.
  \]
  Hence $\tilde h_0 \perp \ker\mathscr K = \ker\mathscr K^\ast$, the
  latter because of self-adjointness, so
  \[
  \tilde h_0 \in (\ker\mathscr K^\ast)^\perp =
  \overline{\operatorname{ran} \mathscr K} = \operatorname{ran}
  \mathscr K
  \]
  because the range of a Fredholm operator is closed. In other words
  there is $F\in L^2_0$ such that $\mathscr K F = \tilde h_0$, and
  thus there is a solution to the equations in
  \cref{treeEquationForH1}.
\end{enumerate}

\subsection{Numerical reconstruction algorithm}
We describe here the numerical implementation of \cref{alg2}. Its
inputs include \verb#tHist#, \verb#K#, the discretized time-variable
vector and cell containing the discretized components of the response
matrix $k=(k_{ij})_{i,j=1}^N$ and \verb#f# which is a vector modeling
the discretized action time function. Other inputs are \verb#a0#,
\verb#A0#, which are vectors containing the wave speed and
cross-sectional area at the boundary points of the network, and
\verb#tau#, \verb#g#, \verb#reguparam# are scalars representing
$\tau$, the constant of gravity $g$, and the Tikhonov regularization
parameter used in the inversion of \cref{eq:solveQ}. Its output,
\verb#Qtau#, is a cell whose components are discretizations of
$t\mapsto Q_p(t,x_j)$ that solved \cref{eq:solveQ} for the various
boundary points $x_j$.
\begin{lstlisting}
  function Qtau = makeHeq1(tHist, K, tau, f, a0, g, A0, reguparam)

  dt = tHist(2)-tHist(1);
  nBndryPts = size(K,1);
\end{lstlisting}

We write first the $N$ (\verb#=nBndryPts#) integral equations
containing the $N$ response function components. This shall be indexed
by \verb#j# (pipe end where pressure measured) and \verb#i# (pipe end
where pulse sent from). Recall the equation,
\[
h_0 = \frac{a}{A(x_j)g} q(t,x_j) + \sum_{i=1}^N \frac12 \int_0^\tau
q(t,x_i) ( k_{ij}(\abs{t-s}) + k_{ij}(2\tau-t-s) ) ds
\]
for $j=1,\ldots,N$ and $\tau-f(x_j) < t \leq \tau$. But recall that we
must still set $q(t,x_j) = 0$ when $t\leq \tau-f(x_j)$. Our numerical
implementation assumes that $h_0=1$.

We will write the equation above as \verb#H*q = RHS# where $q$ is a
block vector where each block is indexed by $i$ and
\[
q^i = [q(dt,x_i); q(2\,dt,x_i); ... ; q(M\,dt,x_i)]
\]
with $M = \lfloor{\tau/dt}\rfloor$. Each component of \verb#RHS# is
either $h_0=1$ or $0$. We do a piecewise constant approximation of all
of these functions and equations. We discretize \verb#t=l*dt#,
\verb#s=k*dt# with \verb#l,k=1,...,M#.
\begin{lstlisting}[firstnumber=last]
  M = floor(tau/dt); % Number of time-steps of length dt in (0,tau).
  % Discard unneeded measurements:
  tHist = tHist(1:2*M);
  for i = 1:nBndryPts
    for j = 1:nBndryPts
      K{i,j} = K{i,j}(1:2*M);
    end
  end
  tVec = (1:M)*dt; lVec = 1:M;
  sVec = (1:M)*dt; kVec = 1:M;
  [t,s] = ndgrid(tVec,sVec);
  [l,k] = ndgrid(lVec,kVec);
\end{lstlisting}
Set the tolerance for floating point numbers next. If two numbers are
at most tolerance from each other, then they are considered the
same. Without doing this, some numerical errors might happen when
comparing two floating point numbers that are close to each
other. This would cause errors of the order of $dx$ in the area
reconstruction. We can now start discretizing the integral equation.
\begin{lstlisting}[firstnumber=last]
  tol = dt/4;
  Mij = nan(M,M); % Block (j,i) of matrix H
  H = nan(nBndryPts*M, nBndryPts*M); % The linear operator
  h0j = nan(M,1); % Block of RHS corresponding to j
  RHS = nan(nBndryPts*M,1); % The RHS vector having 1's or 0's
  for j = 1:nBndryPts
    h0j = ones(M,1);
    h0j((tVec - (tau - f(j)) <= tol)) = 0;
    RHS( (j-1)*M + (1:M) ) = h0j;
    for i = 1:nBndryPts
      % Use 2*M-(l-1)-(k-1)=2*M+2-l-k for the time-reversal.
      Mij = 0.5 * dt * ( K{i,j}(1+abs(l-k)) + K{i,j}(2*M+2-l-k) );
\end{lstlisting}
Recall that $q(s,x_i)$ must be zero when $s \leq \tau-f(x_i)$. Hence
the columns of the matrix that hit the indices corresponding to
$s\leq\tau-f(x_i)$ should be zero. Similarly, when $t \leq \tau -
f(x_j)$ we want the equation to give $q(t,x_j) = 0$, so the part
coming from the integral must be zeroed. Then add the
``identity''-type part of the integral equation, and save the block
into the final matrix.
\begin{lstlisting}[firstnumber=last]
      Mij((s - (tau - f(i)) <= tol)) = 0;
      Mij((t - (tau - f(j)) <= tol)) = 0;
      if i==j
        Mij = Mij + a0(j)/(A0(j)*g) .* eye(M);
      end
      H( (j-1)*M + (1:M), (i-1)*M + (1:M) ) = Mij;
    end
  end
\end{lstlisting}
Next, we solve the integral equation for $q$ using Tikhonov
regularization. The simplest way would be to set\\
\hspace*{\stretch{1}}{\scriptsize
  \verb#q = [H; sqrt(reguparam)*eye(size(H))] \ [RHS; zeros(size(RHS))];#}\hspace*{\stretch{1}}\\ but
it is unnecessarily slow. For a faster solution we first remove all
the rows and columns which would look like $0=0$. The vector \verb#nz#
shows the indices that are not forced to be zero, i.e. those where
$\tau - f(x_i) < s$. Then we remove the corresponding rows and columns
from $H$. Then solve the equation.
\begin{lstlisting}[firstnumber=last]
  nz = false(nBndryPts*M,1);
  for i=1:nBndryPts
    nz( (i-1)*M + (1:M) ) = ((sVec - (tau - f(i))) > tol);
  end
  H = H(nz, nz);
  qnz = [H; sqrt(reguparam)*eye(size(H))] ...
    \ [RHS(nz); zeros(size(RHS(nz)))];
  q = zeros(nBndryPts*M,1);
  q(nz) = qnz;
\end{lstlisting}
Finally split the vector \verb#q# into a cell whose components
correspond to the boundary flows at the different boundary points.
\begin{lstlisting}[firstnumber=last]
  Qtau = cell(nBndryPts, 1);
  for i=1:nBndryPts
    Qtau{i} = q( (i-1)*M + (1:M) );
  end

  end
\end{lstlisting}

\makeatletter
\renewcommand{\@biblabel}[1]{[#1]}
\makeatother
\bibliography{./network_area_reconstruction} %{refs}
\bibliographystyle{mine}

\end{document}